\documentclass[11pt,leqno]{article}
\usepackage{amsthm,amsfonts,amssymb,epsfig,graphics,amsmath,oldgerm}
\usepackage[latin1]{inputenc}\relax

\newcommand{\RR}{{\mathbb R}}
\newcommand{\NN}{{\mathbb N}}

\newcommand{\CC}{{\mathbb C}}
\newcommand{\EE}{{\mathbb E}}

\newcommand{\bK}{\mathbf{K}}
\newcommand{\bR}{\mathbf{R}}
\newcommand{\bS}{\mathbf{S}}

\newcommand{\cA}{\mathcal{A}}

\newcommand{\cE}{\mathcal{E}}

\newcommand{\cL}{\mathcal{L}}
\newcommand{\cM}{\mathcal{M}}

\newcommand{\cO}{\mathcal{O}}
\newcommand{\cP}{\mathcal{P}}

\newcommand{\cS}{\mathcal{S}}

\font\tenronde=rsfs10
\font\sevenronde=rsfs7
\font\fiveronde=rsfs5
\newfam\rondefam
\scriptscriptfont\rondefam=\fiveronde
\textfont\rondefam=\tenronde
\scriptfont\rondefam=\sevenronde
\def\ronde{\fam\rondefam\tenronde}

\newcommand{\sF}{{\ronde F}}

\newcommand{\sT}{{\ronde T}}

\newcommand{\eps}{\varepsilon }
\newcommand{\vp}{\varphi}
\newcommand{\D}{\partial }

\newcommand{\na}{{\nabla}}
\newcommand{\Id}{\mathrm{Id}}
\newcommand{\re}{\mathrm{Re\,}}
\newcommand{\im}{\mathrm{Im\,}}

\newcommand{\mez}{\frac{1}{2}}

\newtheorem{theo}{Theorem}[section]
\newtheorem{prop}[theo]{Proposition}
\newtheorem{cor}[theo]{Corollary}
\newtheorem{lem}[theo]{Lemma}
\newtheorem{defi}[theo]{Definition}
\newtheorem{ass}[theo]{Assumption}

\newtheorem{rem}[theo]{Remark}

\numberwithin{equation}{section}

\title{ $L^2$ well posed   Cauchy Problems and   Symmetrizability of 
First Order Systems}

\author{
{\textsc{ Guy Métivier} }\footnote {\small Université de Bordeaux - CNRS, Institut de Mathématiques de Bordeaux, 
351 Cours de la Libération , 33405 Talence Cedex, France}}

\begin{document} 
\maketitle

\begin{abstract}
The Cauchy problem for first order system $L(t, x, \D_t, \D_x)$ is known to be well posed in $L^2$ when a it admits 
a microlocal symmetrizer $S(t,x, \xi)$ which is smooth in $\xi$ and Lipschitz continuous in  $(t, x)$. This paper contains 
three main results. First we show that a Lipsshitz smoothness globally in $(t,x, \xi)$ is sufficient. 
Second, we show that the existence of symmetrizers with a given smoothness is equivalent to the existence 
of  \emph{full symmetrizers} having the same smoothness. This notion was first introduced in \cite{FriLa1}. 
This is the key point to prove the third result that the existence of microlocal symmetrizer is preserved if one changes 
the direction of time, implying local uniqueness and finite speed of propagation. 
\end{abstract} 

\tableofcontents 

\section{Introduction}

 This  paper is concerned with the well posedness in $L^2$ of  the Cauchy problem for first order square systems
  \begin{equation}
 \label{eqt}
 \left\{\begin{aligned}
  &Lu :=  A_0(t, x) \D_t u  +   \sum_{j = 1}^d A_j (t,  x) \D_{  x_j} u  + B(t,x) u  = f , \qquad   t > 0, 
   \\
   &u_{| t = 0} = u_0. 
 \end{aligned}\right.
 \end{equation} 
 The starting point is the well known theory of hyperbolic symmetric systems in the sense of 
 Friedrichs (\cite{Fried1, Fried2}): if the matrices $A_j$ are Lipschitz continuous on $[0,T] \times \RR^d$, 
 hermitian symmetric, and if $A_0$ is definite positive with $A_0^{-1}$ bounded, then for all 
 $u_0 \in L^2(\RR^d)$ and $f  \in L^{1} ([0, T] ; L^2(\RR^d))$, 
 the equation \eqref{eqt} has a unique solution 
 $u \in C^{0} ([0, T] ; L^2(\RR^d))$ which satisfies 
   \begin{equation} 
\label{estim2}
   \big\|  u (t) \big\|_{L^2(\RR^d)}   \le C \big\|  u _0  \big\|_{L^2(\RR^d)}
  + C \int_0^t  \big\| L    u  (s) \big\|_{L^2(\RR^d)} ds, 
\end{equation}
for some constant  $C$ independent of $u_0$.  Additional properties are  local uniqueness
and finite speed propagation.  The question discussed in this paper is to know for which systems
these properties remain true.

For scalar equations of order $m$, the analogue would be the well posedness in  Sobolev spaces 
$H^{m-1}$, for which strict hyperbolicity is  necessary (\cite{IP}) and sufficient   (\cite {Gard2, Leray}).
 This completely settles the question   for scalar equations but for systems, the situation 
 is much more complex.

A necessary condition has been given by V.Ivrii and V.Petkov (\cite{IP}): they have shown 
that  if the estimate \eqref{estim2} is valid for $u \in C^\infty_0 (]0, T[ \times \RR^d)$,   
 then there exists a  bounded microlocal symmetrizer $S(t,x, \xi) $ for \eqref{eqt} (the precise definition is recalled below). This is equivalent to a strong form of hyperbolicity of the principal symbol, which we call   \emph{strong hyperbolicity of the symbol}, namely that $L + B_1$  is hyperbolic for all matrix $B_1(t,x)$. 
Of course it is stronger than hyperbolicity which is  known to be a necessary condition for the Cauchy problem to be well posed 
in $C^\infty$ 
(see \cite{Lax1, Miz1, Gard1}  and the review paper \cite{Gard3}). In particular,  
\eqref{estim2}  are the best estimates in terms of regularity that one can expect for the Cauchy problem.

  On the side of sufficient conditions, except in the constant coefficient case, 
  where the energy estimate \eqref{estim2} is easily obtained on the space-Fourier transform of the equation, 
  the existence of a bounded symmetrizer  does not imply in general   that the problem is well posed, even in $C^\infty$,   
A counterexample is given in  \cite{Str} and another one  is proposed in Section~\ref{sectionExamples}.   
 Besides the case of symmetric systems recalled above for which the symmetrizer $S(t, x)$ is independent of $\xi$, 
  the Cauchy problem is known to well posed in $L^2$ 
when the microlocal symmetrizer is smooth in $\xi$ and at least 
Lipschitz continuous in $(t,x)$ (see \cite{Lax2, Met} and Theorem~\ref{L2WPLipold} below for a precise statement). 
In this case, the energy estimates are proven using 
the usual pseudo-differential calculus or the para-differential calculus when the coefficient 
have limited smoothness. This covers the case of  strictly hyperbolic systems and
the more generally case of hyperbolic systems with constant multiplicity (e.g. \cite{Cal}, \cite{Yam}). 
This also applies to a the case of "generic" double eigenvalues, still assuming the strong hyperbolicity 
of the symbol, see Theorem~\ref{theodouble} below.

The first objective of this paper is to revisit these questions under the angle of the smoothness 
of the symmetrizer. We prove that 
the Lipschitz continuity in $(t,x,\xi)$ for $\xi \ne 0$  of the symmetrizer $S$ is sufficient to
 obtain the $L^2$ estimates and the $L^2$ well posedness. In addition, we give examples and counterexamples showing that the Lipschitz condition is sharp. 
 
 The second main result of this paper is to  prove that the existence of microlocal symmetrizer 
is preserved by  a change of time,  as this  is essential to obtain local uniqueness
and the precise description of the propagation of the support of solutions (see \cite{JMR1,Rauch1}). 
More surprisingly, we show that the existence of symmetrizers of a given smoothness
is equivalent to the existence  of \emph{full symmetrizers} of the same smoothness, a notion introduced in \cite{FriLa1}. 
This link is the key point in the proof of existence of of microlocal symmetrizers in any direction of hyperbolicity. 

\bigbreak

We now briefly  present the results.  Note that \eqref{estim2} applied to 
$e^{ \gamma t} u$, $u \in C^\infty_0 (]0, T[ \times \RR^d)$   implies that 
\begin{equation} 
\label{estim1}
\forall \gamma \ge \gamma_0, \  :  \quad 
\gamma \big\|  u \big\|_{L^2(\RR^{1+d})}   \le C  \big\| (L + \gamma A_0)  u \big\|_{L^2(\RR^{1+d})} , 
\end{equation}
 for some constants $C$ and $\gamma_0$ independent of $ u$. 
 This estimate is elliptic like and applying it to functions of the form
 \begin{equation*}
 u (\tilde x)  = e^{ i \lambda \tilde x \cdot \tilde \xi    }   \lambda^{ - \alpha d/2} \chi (\lambda^\mez (\tilde x - \tilde x_0))   \underline u, 
 \end{equation*}
 and  $\gamma = \lambda \gamma_0$  and letting $\lambda$ tend to $+ \infty$ implies 
  \begin{lem}
 Suppose that the coefficients of $L$ are continuous and bounded on the open set $ \Omega$ and there are constants $\gamma_0$ and $C$ such that 
  \begin{equation}
  \label{P2b}
 \gamma \big\|  u  \big\|_{L^2 ( \Omega) }  \le   C 
   \big\| \big( L   + \gamma A_0)\big) u   \big\|_{L^2(\Omega)} .  
 \end{equation} 
 for all $\gamma \ge \gamma_0$ and $u \in C^\infty_0(  \Omega)$. 
Then the principal symbol $L_1( \tilde x, \tilde \xi)$  of $L$ satisfies 
for all  $(t,x) \in \Omega$, all $\gamma \in \RR$ and $u\in \CC^N$: 
  \begin{equation}
  \label{estresolv}
  | \gamma |  \big|   u \big|  \le  
   C  \big|    ( L_1(t,x, \tau,  \xi)  +  i    \gamma A_0(t,x) )   \underline u   \big|.  
 \end{equation}
 \end{lem}
 There is no sign condition on $\gamma$ as seen by   changing $\tilde \xi$ to $-\tilde \xi$.   
 When it holds, we say that the \emph{symbol is strongly hyperbolic} in the time direction. 
The condition \eqref{estresolv} has several equivalent formulations,  see Section~\ref{sectionstroghyp}.
One of them is that the symbols admits a bounded symmetrizer.  

\begin{defi}
\label{defmicrosym}
 A microlocal symmetrizer for $L_1$ is a  bounded  
 matrix $S(t, x, \xi)$, homogeneous of degree $0$ in $\xi\ne 0 $,  such that  $S(t, x, \xi) A_0(t, x)$
 is symmetric and uniformly definite positive, and   $S(t, x, \xi) A(t, x, \xi) $
  is symmmetric, where 
 $A(t, x, \xi) = \sum A_j(t, x) \xi_j$. 
\end{defi} 
 
 Combining the lemma and Theorem~\ref{theostronghyp} below, we recover  the 
 the necessary condition given in \cite{IP}: 
\begin{prop} 
If $L$ has continuous coefficient on the open set $ \Omega$ and there are constants $\gamma_0$ and $C$ such that \eqref{P2b} is satisfied, then 
the principal symbol $L_1$ must admit  a bounded symmetrizer $S(t,x, \xi)$ on 
$\Omega \times \RR^d\setminus{0}$. 
\end{prop}
 
In the constant coefficients case, the existence of a bounded symmetrizer 
is also sufficient, as immediately seen 
by Fourier synthesis. 
 In general, that is for variable coefficients, this condition is 
 far from being sufficient for the well posedness of the Cauchy problem~\eqref{eqt}: 
 in section~\ref{sectionExamples} we give an example 
 of  a $3 \times 3$ systems in space dimension $d=2$, whose symbol
 $L(x, \tau, \xi)$ is strongly hyperbolic uniformly in $x$, and such that 
 the Cauchy problem \eqref{eqt} is ill posed, even locally and with $C^\infty $ data.

 On the side of sufficient conditions, let us first recall 
 the following result:
 \begin{theo}
\label{L2WPLipold}
Suppose that  the coefficients  $A_j \in W^{1, \infty} ([0, T] \times \RR^{d})$ and  there exists a microlocal 
symmetrizer $S$, homogenenous of degree $0 $ and $C^\infty$in $\xi \ne 0$ which satisfies 
$\D_{t,x}^\beta \D_\xi^\alpha S \in L^\infty ([0, T] \times \RR^d \times S^{d-1})$ for all 
$\alpha \in \NN^d$ and all $| \beta | \le 1$. 
    Then, there are constants 
$C$ and $\gamma$ such that  for all 
  $u_0 \in L^2(\RR^d)$ and $f \in L^1([0, T] \times \RR^d)$, the Cauchy problem 
\eqref{eqt}  has a unique solution 
$u \in C^0([0, T] ; L^2(\RR^d))$ which satisfies  
\begin{equation} 
\label{estim2b}
  \big\|  u (t ) \big\|_{L^2(\RR^d)}   \le C e^{ \gamma_0 t}  \big\|  u _0 \big\|_{L^2(\RR^d)}
  + C \int_0^t e^{ \gamma_0 (t-s)}   \big\| L    u  (s) \big\|_{L^2(\RR^d)} ds. 
\end{equation}
\end{theo} 
 
 When the symmetrizer does not depend on $\xi$, this is Friedrichs theory, in which case 
 the estimate \eqref{estim2b} is easily obtained by forming  the real part of the scalar product of $S Lu$
  with $u$ and performing integrations by parts. 
  For microlocal symmetrizers, one replaces the multiplication by  $S $ by  the action of the 
  pseudodifferential operator 
  $S(t,x, D_x)  $ (\cite{Lax2}) when the coefficients are also smooth in $x$, or by a paradifferential version 
  when the coefficients are Lipschitz (see eg \cite{Met}). This theorems applies to hyperbolic systems with constant multiplicities   which admit   smooth symmetrizers. 
Indeed,   multiple eigenvalues of $A(t,x, \xi)$ with variable multiplicities are the 
main difficulty  for the  construction of smooth symmetrizers. However, 
we note in Proposition~\ref{theodouble}  that the theorem above applies to 
strongly hyperbolic systems which have only generic double eigenvalues. 
  
  The first main result of the paper extends this result to Lipshitz symmetrizers, 
  using a Wick quantization of the symbols. 
  
 \begin{theo}
\label{L2WPLip}
Suppose that  the coefficients  $A_j \in W^{2, \infty} (\RR^{d+1})$ and  there exists a microlocal 
symmetrizer, homogenenous of degree $0 $ in $\xi \ne 0$ and Lipschtiz 
continuous in $(t, x, \xi)$ on $\RR^{d+1} \times S^{d-1}$. Then, there are constants 
$C$ and $\gamma$ such that  for all 
  $u_0 \in L^2(\RR^d)$ and $f \in L^1([0, T] \times \RR^d)$, the Cauchy problem 
\eqref{eqt}  has a unique solution 
$u \in C^0([0, T] ; L^2(\RR^d))$ which satisfies \eqref{estim2b}. 
\end{theo} 
  
This theorem is proved in Section~\ref{sectionL2WP}.  In  Section~\ref{sectionExamples} we discuss the existence 
of  Lipschitz symmetrizers. In particular, we give examples of systems which admit a Lipschitz  symmetrizer 
but no $C^1$ symmetrizer. We also prove that the Lipschitz  condition is sharp, 
  in the sense that for all $\mu< 1$, there are examples of systems admitting H\"older continuous symmetrizers of order $\mu< 1$, for which the Cauchy problem with  $C^\infty$ data is  is locally ill posed.

\bigbreak
The second part of the paper is concerned with the 
local theory of the Cauchy problem and the finite speed propagation property 
for the support of the solutions. 
A classical proof of this  property relies on the invariance of the assumptions 
by changes of time variables, so that one convexify the initial surface, 
The existence of a local symmetrizer is clearly invariant by change of time,  
  as well as   strict hyperbolicity  or the property that the characteristic variety is smooth 
   with constant multiplicities. In all this case the local theory was well established. 
   This invariance is not clear for the existence of smooth microlocal symmetrizers. 
   However, when there are smooth symmetrizers,     local uniqueness and  finite speed of propagation 
 are proved in \cite{Rauch1}
using  another approach bases on finite difference approximation schemes and uniform estimates due to 
\cite{Lax-Nir, Vaillancourt}. 

The second main theorem of this paper asserts that the  existence 
of a Lipschitz symmetrizer [resp. $C^\infty$]  is preserved by change of timelike directions.  
This is  a key step for establishing a local theory, starting with local uniqueness, 
finite speed of propagation and endind with the sharp description of the
propagation of support  as stated in \cite{JMR1, Rauch0, Rauch1}. 

Let $\tilde x $ denote the space-time variables  $(t, x)$ and set accordingly
$\tilde \xi = (\tau, \xi)$ Assuming that $L_1(\tilde x , \tilde \xi) $ is hyperbolic in the time direction 
$(1,0) \in \RR^{1+d}$, denote by $\Gamma_{\tilde x}$ the cone of hyperbolic directions 
that is the component of $(1,0)$ in $\{ \tilde \xi : \det L_1(\tilde x,  \tilde \xi) \ne 0\} $.

 \begin{theo}
 \label{invt}
 Suppose that the coefficients   $A_j $ are Lipschitz continuous [resp $C^\infty$]  on $\RR^{d+1}$ and that there exists a microlocal 
symmetrizer $S(t, x, \xi)$, homogenenous of degree $0 $ in $\xi \ne 0$ and Lipschtiz 
continuous in [resp $C^\infty$] $(t, x, \xi)$ on $\RR^{d+1} \times S^{d-1}$.   
then for any time-like direction $\tilde \nu = \in \Gamma_{t,x}$, the symbol 
$L(t, x, \tilde \nu)^{-1} A(t,x, \xi)$ admits a Lipschitz [resp $C^\infty$] symmetrizer. 
 
\end{theo}

\begin{cor}
Under the assumptions of Theorem~$\ref{L2WPLip}$, the Cauchy problem for $L$ with 
initial data on any space like hyperplane is well posed in $L^2$. 
 
\end{cor}

 As said above, together with Theorems~\ref{L2WPLipold} and \ref{L2WPLip}, this implies 
 local uniqueness, and finite speed of propagation.  Together with the Lipschitz dependence 
 of the cone of propagation implied by Proposition~\ref{proplipvp}, this allows to apply 
 the results on the precise propagation of support stated in  \cite{JMR1, Rauch0, Rauch1}. 
 We refer the reader to  these papers for precise statements.

The proof of this theorem is based on an intrinsic characterization  of the existence of 
  Lipschitz symmetrizers which uses the notion of {\sl full symmetrizers} introduced by 
  K.O.Friedrichs and P.Lax   \cite{FriLa1}: 
  
\begin{defi}
 A full symmetrizer for \eqref{eqt} is a  bounded  
 matrix $\widetilde S(t, x, \tau, \xi)$, homogeneous of degree $0$ in $(\tau, \xi)\ne 0 $,  such that  
 $\widetilde S(t, x, \tau, \xi) L(t, x,\tau,  \xi) $
  is symmmetric. 
  
  $\widetilde S$ is said to be positive in the direction $\tilde \nu$,  if
  $ \re \widetilde S(t, x, \tau, \xi) L(t, x, \tilde \nu)$ is definite positive 
  on $\ker L(t, x, \tau, \xi)$ for all $(t,x, \tau, \xi)$. 
\end{defi} 
Of course the condition is nontrivial only near characteristic points, but says nothing about 
hyperbolicity.  Our third main theorem is the following;

\begin{theo}
Suppose that $L$ is hyberpolic in  the time direction. Then, $L$ admits a continuous [resp. Lipschitz]
microlocal symmetrizer $S(t,x, \xi)$, if and only if it admits a continuous [resp. Lipschitz]
full  symmetrizer $\widetilde S (t, x, \tau, \xi)$ which is positive in the time direction. 

In this case, the    $\widetilde S$ is  positive in an direction of hyperbolicity $\tilde \nu$.

\end{theo}


\section{Lipschitz symmetrizability is sufficient for the $L^2$ well posedness}
\label{sectionL2WP}

 The goal of this section is to prove Theorems~\ref{L2WPLip}. 
 We consider a system 
  \begin{equation}
 \label{geq}
 L u = \sum_{j = 0}^d A_j ( \tilde x) \D_{\tilde x_j} u     
 \end{equation} 
  with coefficients $A_j$  which are at least $W^{1, \infty} ([0; T] \times \RR^d)$.

\subsection{Wave packets and localization} 
 
  For $u \in L^2 (\RR^s)$, $\lambda> 0$  and $B \in L^\infty (\RR^d \times \RR^d)$, let
 \begin{equation}
 \label{defWB}
 W_{\lambda, B}  u (x, \xi) =    \frac{1}{ (2 \pi ) ^{ \frac{ d}{2}  } }
\left(\frac{ \lambda}{\pi} \right)^{  \frac{d}{4}}   \int e^{ i (x - y) \xi  - \mez \lambda | x- y|^2}
 B(x, y) u (y) dy 
 \end{equation}

 \begin{lem}
 \label{action1} 
  The operator $W_{\lambda, B}$ is bounded from $L^2(\RR^d)$ to 
  $L^2(\RR^{2d})$ and 
 \begin{equation}
 \big\| W_{\lambda, B}  u  \big\|_{L^2(\RR^d \times \RR^d)}  
 \le   \big\| B \big\|_{L^\infty}    \big\| u \big\|_{L^2 (\RR^d)}. 
 \end{equation}
 Moreover, if $ B (x, z) \equiv \Id$,   $W_\lambda := W_{\lambda, \Id} $ is isometric from 
 $L^2(\RR^d)$ into $L^2(\RR^d \times \RR^d)$. 
 \end{lem}
\begin{proof}
Let $\sF$ denote the Fourier transform and 
$ 
v_x (y) = B(x, y)  e^{ - \mez \lambda | x - y |^2}   u (y) . 
$. Then 
$$
W_{\lambda, B} u (x, \xi)  =   \frac{1}{ (2 \pi ) ^{ \frac{ d}{2}  } }
\left(\frac{ \lambda}{\pi} \right)^{  \frac{d}{4}} 
 e^{ i \lambda x \xi}  \sF \big( v_x \big) ( \xi) . 
$$
Therefore 
$$
\begin{aligned}
 \int \big| W_B u (x, \xi) \big|^2  dx d \xi   = 
 \left(\frac{ \lambda}{\pi} \right)^{  \frac{d}{2}}  \int \big| v_x(y) \big|^2 dx dy 
\\ 
= \left(\frac{ \lambda}{\pi} \right)^{  \frac{d}{2}}  
 \int   e^{ - \lambda | x - y |^2}   \big| B(x, y)  &  u (y) \big|^2  dx dy 
 \le \big\| B \big\|^2_{L^\infty} 
    \ \big\| u \big\|^2_{L^2 (\RR^d)}.   
\end{aligned}
$$
When $B = 1$, the inequality is an equality. 
\end{proof}
 
We will adapt the scale $\lambda$ to the size of the 
 frequency $| \xi |$.  One has 
$$
\begin{aligned}
 W_\lambda  u (x, \xi)    =  
     (2 \pi ) ^{ - d  } 
\left(\frac{ 1} {\pi \lambda} \right)^{  \frac{d}{4}}  \int e^{ i x \eta - \frac{1} {2\lambda} | \xi - \eta |^2}     \hat u (\eta) d \eta.
\end{aligned}
$$
 This shows that for a fixed  $\xi$, $W_\lambda u (\, \cdot \, ,  \xi)$  is the inverse Fourier transform of 
 $ 
 w_\xi (\eta) =   \left(\frac{ 1} {\pi \lambda} \right)^{  \frac{d}{4}}  \ e^{  - \frac{1} {2\lambda} | \xi - \eta |^2}     \hat u (\eta) . 
 $ 
 Therefore
 \begin{equation}
 \label{rel2}
 \begin{aligned}
\int  \big|  W_\lambda u (x, \xi) \big|^2 dx  & = (2\pi)^{- d} \int \big| w_\xi (\eta) \big|^2 d \eta 
\\
& =  (2\pi)^{- d}   \left(\frac{ 1} {\pi \lambda} \right)^{  \frac{d}{2}} 
\int   e^{  - \frac{1} { \lambda} | \xi - \eta |^2}    |  \hat u  (\eta) |^2 d \eta. 
 \end{aligned}
 \end{equation}
 Integrating in $\xi$, we recover the isometry of $W_\lambda$,
 but  the important point is that we use \eqref{rel2}   to localize in  
$| \xi |$. 
 
 Consider a dyadic partition of unity
 \begin{equation}
 1 = \vp_0 (\xi)  +  \sum_{j= 1}^\infty \theta_j (\xi) 
 \end{equation}
 with $\vp_0 \in C^\infty_0(\RR^d)$, supported in $\{ | \xi | \le 2 \}$ and equal to one 
 on $\{| \xi | \le 1\}$, 
 $\theta_j (\xi) = \vp_j (\xi) - \vp_{j-1} (\xi)$ and 
 $\vp_j(\xi ) = \vp_0 (2^j \xi)$ for $j \ge 1$. 
 To unify notations, we set $\theta_0 = \vp_0$ and for $j \ge 0$ we denote by 
 define $\Theta_j  $  the operator 
 $$
 \Theta_j u = \sF^{_1} \big( \theta_j \hat u\big),  
 $$
 so that 
 \begin{equation}
u  =   \sum_{j= 0}^\infty \Theta_j  u 
 \end{equation}

 \begin{prop}
 \label{proplocal1}
 For all $n$, $m$ and $\alpha$, there is a constant $C$ such that for all $j \ge 0$
 \begin{equation}
 \big\| | \xi|^m (1 - \vp_{j+2}) W_{2^j} 
 \D_y^\alpha \Theta_j u \big\|_{L^2(\RR^{2d})}  \le  
 C 2^{ -j n}   \big\|  \Theta_j u \big\|_{L^2(\RR^d)} . 
 \end{equation}
 \end{prop}

 \begin{proof}
 By \eqref{rel2} 
  \begin{equation*}
 \begin{aligned}
 \big\| (1 - \vp_{j+2}) W_{2^j}   \D_y^\alpha & \Theta_j u \big\|^2_{L^2(\RR^{2d})}     =   (2\pi)^{- d}   \left(\frac{ 1} {\pi 2^j} \right)^{  \frac{d}{2}} 
\\
& \int   e^{  - 2^{-j} | \xi - \eta |^2}   (1 - \vp_{j+2}(\xi))^2 | \eta^\alpha |^2
(\theta_j (\eta))^2 |  \hat u  (\eta) |^2 d \eta d \xi . 
 \end{aligned}
 \end{equation*}
On the support of $ (1 - \vp_{j+2}(\xi)) 
\theta_j (\eta)$, one has $ |\xi | \ge 2^{j+2} $, $| \eta | \le 2^{j+1}$ so that 
  $| \xi - \eta | \ge \mez | \xi| $   and therefore
$
2^{-j} | \xi - \eta |^2  \ge   \mez 2^{-j} | \xi - \eta |^2  +   \mez  | \xi | . 
$
Hence, 
 \begin{equation*}
 \begin{aligned}
 \big\| 2^{jn}  | \xi|^m  (1 - \vp_{j+2}) W_{2^j}  \D_y^\alpha \Theta_j u \big\|^2_{L^2(\RR^{2d})}   
   \le    (2\pi)^{- d}   & \left(\frac{ 1} {\pi 2^j} \right)^{  \frac{d}{2}} 
\\
 2^{j  n } 2^{( j+1) | \alpha |}  e^{ - 2^{j} }   \int    | \xi|^{2m}  e^{ - \frac{1}{2} | \xi |}  
 e^{  - \mez 2^{-j} | \xi - \eta |^2}   &
| \theta_j (\eta)  \hat u  (\eta) |^2 d \eta d \xi 
\\
&  \le  C     \big\|  \Theta_j u \big\|_{L^2(\RR^d)}^2. 
 \end{aligned}
 \end{equation*}
 \end{proof}
 
 \begin{cor}
 \label{cor103}
 \begin{equation}
 \big\| (1 - \vp_{j+2}) W_{2^j}\Theta_j  \D_{x_k}  u \big\|_{L^2(\RR^{2d})}  \le  
 C    \big\|  \Theta_j u \big\|_{L^2(\RR^d)} . 
 \end{equation}
 \end{cor} 
 
 \subsection{The main estimate} 
 
 Let 
\begin{equation}
A(x, \D_x) = \sum_{j=1}^d A_j(x) \D_{x_j}, \qquad A(x, \xi) = \sum_{j=1}^d \xi_j A_j(x) . 
\end{equation}
 We assume that we are given a matrix $S(s, \xi)$, homogeneous of degree $0$ in $\xi$ such that 
 $S(x, \xi) A(x, \xi)$ is hermitian symmetric. 
 
 Let  $\psi  \in C^\infty_0 (\RR^d)$,  and for 
 $\lambda \ge 1$ introduce  $\psi_ \lambda (\xi) = \psi   (\lambda^{-1} \xi)$.

\begin{prop}
\label{mainestloc}
Suppose that  the coefficient $A_j $ belong to $W^{2, \infty} (\RR^d)$ and that 
$S \in W^{1, \infty} (\RR^d \times S^{d-1})$.  Then, there is a constant $C$ such that for all 
$\lambda \ge 1 $ and $u$ 
 \begin{equation}
 \label{energy5}
 \Big|  \re   \Big( \psi_{\lambda } S   W_{\lambda}   u ,    W_{\lambda} \widetilde A   u  \big)_{L^2(\RR^{2d})} \Big| 
 \le C  \big\|  u \|^2_{L^2} . 
 \end{equation}

\end{prop} 

\begin{proof}
Let $S_\lambda = \vp_\lambda S$ and 
 consider the self adjoint operator 
$
 \Sigma_\lambda =  W_\lambda^* S _\lambda W_ \lambda$: 
$$
 \Sigma_\lambda u (x) = \kappa  \lambda^{  \frac{d}{2} }
 \int e^{ \Phi_\lambda   (x, y, z, \xi)}  S_\lambda ( z, \xi) u (y) dz d\xi d y    
$$
where $\kappa$ is a normallization factor  and 
$$
\Phi_\lambda (x, y, z \xi) = i (x - y) \xi - \mez \lambda (| x - z|^2 + | y-z|^2) .
$$

 The   estimate to prove is 
 \begin{equation}
\label{commut}
\Big| \re  \big(  A(x, \D_x)^*  \Sigma_\lambda u , u   \big)_{L^2} \Big|  \le C \big\| u \big\|_{L^2}^2.   
 \end{equation}
One can replace  $A(x, \D_x)^*$ by $ \tilde A^* := \sum - A_j^* \D_{x_j}$ since the difference 
is bounded in $L^2$ with norm $ O( \sup_j \| A_j \|_{W^{1, \infty}} ) $. One has 
$$
A_j^* \D_j  \Sigma_\lambda u (x)  = \kappa
\int e^{ \Phi_\lambda}  A_j^* (x) \big( i \xi_j - \lambda   (x_j - z_j) \big) S_\lambda (z, \xi) u (y) dz d\xi d y. 
$$
Therefore
\begin{equation}
\label{comm36}
\tilde A^* (x, \D_x ) \Sigma_\lambda =  W_\lambda ^* ( -i A^* S_\lambda) W_\lambda   + \sum_j  ( R_j^1 +  R_j^2 ) 
\end{equation}
where 
$$
R_j^1 u(x) =  i \kappa \int e^{ \Phi_\lambda}  \xi_j \big(A_j^* (z) - A_j^*(x)) \big) S_\lambda (z, \xi) u (y) dz d\xi d y, 
$$
$$
R_j^2 u(x) =  \kappa  \int e^{ \Phi_\lambda} \lambda  (x_j - z_j) A_j^*(x)  \big) S_\lambda (z, \xi)  u (y) dz d\xi d y.
$$
Because $S$ is a symmetrizer for $\sum A_j \xi_j$, the matrix 
$i A^*(x, \xi) S_\lambda$ is skew symmetric and thus the real part 
of the first term  in \ref{comm36} vanishes and it is sufficient to 
show that the remainders $R_j^{1, 2}$ are bounded in $L^2$. Write 
\begin{equation}
A_j^* (x) - A_j^*(z) = \sum_k    (x_k - z_k) \tilde A_{j,k}(x, z), 
\end{equation}
with $\tilde A_{j,k} \in W^{1, \infty}$, and use that 
$$
\begin{aligned}
2 \lambda (x_k - z_k)  = \D_{z_k} \Phi_\lambda  - i \lambda \D_{\xi_k} \Phi_\lambda := Z_k \Phi_\lambda 
\end{aligned}
$$
Integrating by parts in $(z, \xi)$ yields that $R_j^1 = \sum_k R^1_{j, k} $  with 
$$
R_{j,k}^1 u(x) =  - \frac{i \kappa }{\lambda}    \int e^{ \Phi_\lambda}  Z_k^* 
\big( \xi_j  A_{j,k}^*   S_\lambda \big)  u (y) dz d\xi d y.  
$$
Note that  
$
\frac{1}{\lambda} Z_k^* 
\big( \xi_j  A_{j,k}^*   S_\lambda \big)   $ is a sum a terms of the form $   B_{l } (x, z) S_l (z, \xi, \lambda ) 
$ 
where the $B_l$ and $S_l$ are uniformly bounded. Therefore 
    $R_{j,k}^1$ is of the sum of the operators 
 $W^*_{\lambda, B_{l}} S_l W _\lambda $   where the definition of 
$W_{\lambda, B_l}$ is   given in \ref{defWB}. 
Lemma~\ref{action1}  implies that the $R_j^1$  are uniformly in $\lambda$, 

The analysis of  $ R_j^2  $ is similar and  the proof of the proposition is complete. 
 \end{proof} 
 
 
 \subsection{Proof of  Theorem~\ref{L2WPLip} }
 
 Suppose that  the operator  $L$ in \eqref{geq} has $W^{2,  \infty} $ coefficients. 
 Without loss of generality, multiplying $L$ on the left by $A_0^{-1}$, 
 we assume that the coefficient of $D_t $ is $A_0 = \Id$ so that 
 $L = \D_t + A(t, x, \D_x)$. 
 We are given a Lipschtiz symmetrizer 
 $S (t, x, \xi)$  which  is  uniformly definite positive and such that 
 \begin{equation}
 S, \ \D_{t, x } S,  \  | \xi | \D_\xi S  \in L^\infty.
 \end{equation}  
 Consider the energy
 \begin{equation}
 \label{energy}
 \cE_t  (u)  = \sum_{j=0}^\infty  \big(  S(t)   W_{2^j} \Theta_j u ,    W_{2^j} \Theta_j u  \big)_{L^2(\RR^{2d})}. 
 \end{equation}
 
 \begin{lem}
 There are constants $C \ge c > 0$ such that 
 \begin{equation*}
 c \big\| u \big\|_{L^2(\RR^d)}^2  \le \cE (u)  \le  C\big\| u \big\|_{L^2(\RR^d)}^2 . 
 \end{equation*}
 \end{lem}
 \begin{proof}
 Because $S(t, x, \xi) $ is definite positive bounded from above and from below, 
  \begin{equation*}
   \cE_t  (u) \approx \sum_{j=0}^\infty  \big\|   W_{2^j} \Theta_j u \big\|^2_{L^2(\RR^{2d})}
   =  \sum_{j=0}^\infty  \big\|    \Theta_j u \big\|^2_{L^2(\RR^{2d})} \approx 
    \big\|     u \big\|^2_{L^2(\RR^{2d})}. 
    \end{equation*}
 \end{proof} 
 
 The Theorems follows from the energy estimate
 
 \begin{prop} If $u $ satisfies $L u = f$, 
 then 
 \begin{equation}
 \frac{d}{dt} \cE_t  (u(t)) \le   C   \Big( 
   \big\|    f (t)  \big\|_{L^2(\RR^{2d})}    \big\|     u (t) \big\|_{L^2(\RR^{2d})} +   \big\|     u(t) \big\|^2_{L^2(\RR^{2d})}
   \Big) . 
 \end{equation}
 
 \end{prop}
 
  \begin{proof} 
 One has 
 \begin{equation*}
 \begin{aligned}
  \frac{d}{dt} \cE_t  (u(t)) &  =  (\D_t \cE) (u (t) )  + 
  2 \re \widetilde \cE_t (u(t), \D_t u (t)))
  \\
  &  =  (\D_t \cE) (u (t) )  + 
  2 \re \widetilde \cE_t (u(t), f (t))) - 2 \re \widetilde \cE_t (u(t), A u (t)))
  \end{aligned}
 \end{equation*}
 where $\D_t \cE_t$ is the expression \eqref{energy} with $S$ replaced by $\D_t S$ and 
 $\widetilde \cE$ is the bilinear version of $\cE$. Because 
 $\D_t S \in L^\infty$, the first term 
 is $O (\| u(t) \|_{L^2}^2)$. 
 Similarly, the second term 
 is  $O (\| u(t) \|_{L^2} \| f(t) \|_{L^2})$ and it remains to prove that
 \begin{equation}
 \label{mainestim}
 \big|  \re \widetilde \cE_t \big(u(t), A u (t)\big)  \big\|  \le  
 C \big\| u(t) \big\|_{L^2}^2. 
 \end{equation}  
 For simplicity we drop the time from the notations, $t$ being a parameter 
 and all the estimate below being uniform in $t$. 
 
 The expression to consider is
  \begin{equation}
 \label{energy2}
\widetilde  \cE  (u, A u)  = \sum_{j=0}^\infty  \big(  S   W_{2^j} \Theta_j u ,    W_{2^j} \Theta_j  A u  \big)_{L^2(\RR^{2d})}. 
 \end{equation}
Corollary~\ref{cor103} implies that 
 \begin{equation*}
 \Big|   \big(  (1- \vp_{j+2}) S   W_{2^j} \Theta_j u ,    W_{2^j} \Theta_j  \widetilde A u  \big)_{L^2(\RR^{2d})}
 \Big| \lesssim 
 \sum_k  \big\| \Theta_j u \big\|_{L^2} 
 \big\| \Theta_j (A_k u ) \big\|_{L^2} , 
\end{equation*}
 and the sum over $j$ of these terms is $ O (  \| u \|^2_{L^2}   )$.
 Next, we  replace 
 $\Theta_j   ( A u)  $ by $  A \Theta_j   u $ using the following independent lemma  
 
 \begin{lem}
 \label{lemcomloc}
The $ g_j =  [  A, \Theta_j ]  u$ satisfy
  \begin{equation}
  \label{114}
 \sum \big\| g_j  t)  \big\|^2_{L^2}  \lesssim      \big\|     u (t) \big\|^2_{L^2}. 
 \end{equation}
 \end{lem}

Since $S$ is bounded and since the $W_\lambda$ are isometries one has 
 \begin{equation*}
 \label{energy2b}
\Big| \sum_{j=0}^\infty  \big(\vp_{j+2}    S   W_{2^j} \Theta_j u ,    W_{2^j}  g_j   \big)_{L^2(\RR^{2d})}
\Big|  \lesssim   \sum_{j=0}^\infty  \big\| \Theta_j u \|_{L^2}     \big\|  g_j   \big\|_{L^2}
\lesssim   \big\|  u \|^2_{L^2} . 
 \end{equation*}
Summing up, we have proved that 
\begin{equation*}
\widetilde  \cE  (u, A u) = 
\sum_{j=0}^\infty  \big( \vp_{j+2} S   W_{2^j} \Theta_j u ,    W_{2^j}  A  \Theta_j u  \big)_{L^2(\RR^{2d})} 
 + O( \big\|  u  \big\|^2_{L^2}). 
\end{equation*}
By Proposition~\ref{mainestloc}, the real part of the sum is 
$O (\sum \| \Theta_j u\|_{L^2}^2) = O( \| u \|_{L^2}^2)$ finishing the proof of 
\ref{mainestim} and of the proposition.  
\end{proof}
 
 \begin{proof}[Proof of Lemma~$\ref{lemcomloc}$]
  Using the para-differential calculus (see e.g. \cite{Met}) , one has 
 \begin{equation*}
 A  (x, \D_x)  = T_{ i A }  + R 
 \end{equation*}
 where $T_{i A}$ is a pardifferential  operator  of symbol 
 $i A(x, \xi)$ and $R$ is bounded from $L^2$ to $L^2$. Thus
  \begin{equation*}
 g_j  = [ T_A , \Theta_j ] u   + R \Theta_j u  -  \Theta_j Ru .
 \end{equation*}
 The last two terms  satisfy \eqref{114}.  The symbolic calculus implies that the 
 $[ T_{ i A }, \Theta_j ] $  are uniformly bounded in $L^2$  since the coefficients of 
 $A$ belong to $W^{1, \infty}$. Moreover,  the spectral properties 
 of the paradifferential calculus shows that 
 $ [ T_{i A}  , \Theta_j ]  = [ T_{i A} , \Theta_j ]  \widetilde \Theta_j $  where the cut-off function
$ \tilde \theta_j  $ is supported in the annulus 
$\{   2^{j - n}  \le | \xi | \le 2^{j+n} \} $  for some fixed $n$ if $j \ge 1$
and in the ball  $\{     | \xi | \le 2^{n} \} $    if $j = 0$. Therefore 
$ g'_j = [ T_{i A}  , \Theta_j ] u $  satisfies
$$
\big\| g'_j \big\|_{L^2}  \lesssim \big\| \widetilde \Theta_j u \big\|_{L^2} 
$$
and thus \eqref{114}. 
 \end{proof}



\section{Examples and counterexamples }
\label{sectionExamples}

 In this section, we discuss the question of the existence 
 of  symmetrizers of limited smoothness. The case of generic double eigenvalues is 
 specific, but for eigenvalues of higher order, it is easy to construct 
 examples of systems with symmetrizers which necessarily have no, or a limitied, smoothness. 
 
Second, we show on an example  that that Lipschitz smoothness is sharp, even for well posedness 
in $C^\infty$. 

\subsection{Example of non smooth  symmetrizers}

In space dimension two consider, near the origin,  a system of the form
\begin{equation}
\label{explegen0}
L_0 (x, \D_t, \D_x,  \D_y)   =  \cL_0 (\D_t , \D_x, x \D_y) = \D_t + A  \D_x + x B \D_y, 
\end{equation}
with $\cL(\tau, \xi, \eta)$   strictly hyperbolic. Consider next a perturbation 
\begin{equation}
\label{explegena}
L_a (x, \D_t, \D_x,  \D_y)   =  L_0 (x, \D_t, \D_x,  \D_y)  + x a(x) C \D_y 
= \cL (a(x), \D_t , \D_x, x \D_y)
\end{equation} 
We will give explicit examples below. 
For $a$ small, $\cL(a, \tau,\xi, \eta)$ is still strictly hyperbolic and therefore it 
has   smooth symmetrizers  $\cS(a, \xi, \eta)$ for $(\xi, \eta) \ne (0, 0)$, providing 
bounded symmetrizers for $L_a (x,  \tau, \xi, \eta) $
\begin{equation}
\label{symgena}
S(a, x, \xi, \eta) =  \cS(a, \xi,  x \eta)  
\end{equation}
for $(x, \xi) \ne (0,0) $.  On the unit sphere $\xi^2 + \eta^2 = 1$, they  are smooth   when 
  $(x, \xi) \ne (0,0) $. The definition of $S$ can be extended at  $(x, \xi) = (0,0) $, 
  but  in general they have a singularity there.  
 
 \begin{lem}
 Suppose in addition that $\cL_0$ is symmetric. Then, for $a$ small, there is a symmetrizer 
 $\cS$ of the form 
 \begin{equation}
\label{syma}
\cS (a, \xi, \eta) =  \Id  + a \cS_1(a, \xi, \eta) 
\end{equation}
with $\cS_1$ homogeneous of degree $0$ in $(\xi, \eta)$ and smooth 
in $ (a, \xi, \eta)$  for $(\xi, \eta) $ in unit sphere $\xi^2 + \eta^2 = 1$. 
 \end{lem}
 
 \begin{proof}
 The spectral projectors $\Pi_j(a, \xi, \eta)$ are smooth in 
 $ (a, \xi, \eta)$  for $(\xi, \eta) $ in unit sphere $\xi^2 + \eta^2 = 1$ and 
 $\cS = \sum \Pi_j^* \Pi_j$ is a symmetrizer. Since $\cL_0$ is symmetric, the  $\Pi_j$ are symmetric when $a= 0$ and therefore
$\cS (0, \xi, \eta) =\Id$, implying \eqref{syma}.  
 \end{proof}
 
 Substituting in \eqref{symgena} implies the following
 
 \begin{cor}
 \label{corholdsym}
If $\cL_0$ is symmetric and strictly hyperbolic and $a(x) = | x |^\alpha$ with $0 < \alpha < 1$, 
$L_a$ admits    H\"older continuous 
symmetrizers  $S (a, x, \xi, \eta)$ of class $C^\alpha$. 

 If $a(x) = x$, it admits a Lipschtiz symmetrizer. 
 \end{cor}

 \noindent {\bf Example 1:  }  consider 
 \begin{equation}
 \label{exple1}
 L_a     =  \D_t  + 
 \begin{pmatrix}
 0 & \D_x  + x a  \D_y  & \ x  \D_y 
 \\ \D_x  - x a \D_y & 0 & 0
 \\
 x  (1 + a^2) ) \D_y  & 0 & 0
 \end{pmatrix} 
 \end{equation}  
In this case,  $\det \cL(a, \tau, \xi, \eta) = \tau (\tau^2 - \xi^2 - \eta^2)$, so that 
$\cL_a$ is always strictly hyperbolic.

  \begin{lem}
  If $a \ne 0 $ is a constant, there are bounded symmetrizers  $S ( x, \xi, \eta)$ for $L_a$, but no continuous 
 symmetrizers at  $(x, \xi) = (0, 0) $ when $\eta = 1$.

 \end{lem}

\begin{proof} 
Fix $\eta = 1$.  If $ S (x, \xi) $ is a symmtetrizer, then its complex conjugate is also a symmetrizer, 
  so that $ S + \overline { S}$  is a symmetrizer. Thus,  it is sufficient to consider 
  the case where   $ S$ has real coefficients $s_{j, k}$. The symmetry condition reads 
  \begin{equation*}
\begin{aligned}
 (\xi + a x) s_{11}& =  (\xi - a x) s_{22} + (1+a^2) x s_{23}
\\
 \eta  s_{11} & =  (\xi - a x ) s_{23} + (1+a^2) x s_{33}
 \\
 \eta s_{12} & = (\xi + a x)   s_{13}. 
\end{aligned}
\end{equation*}
  The third condition is independent of the first two, it only involves $s_{12} $ and $s_{13}$, and
  is trivially satisfied by $s_{12}= s_{13} = 0$. 
  
  There is no restriction in assuming that $s_{22} = 1$.  Setting 
  $s'_{11} =   s_{11} -1 $,  $s'_{33} = (1+a^2) s_{33} -s_{11} $, one must have 
   \begin{equation}
\label{pds}
\begin{aligned}
 (\xi + a x ) s'_{11}& =  -2 a x  + (1+a^2) x  s_{23}
\\
x   s'_{33} & =  - (\xi - a x ) s_{23} .
  \end{aligned}
\end{equation}
 Suppose that the coefficients are continuous at $(x,\xi )= (0,0)$. Then taking $x = 0$
 and $\xi \ne 0$ in the equations above, dividing out by $\xi$ and letting $\xi$ tend to $0$ implies that 
 $s_{2,3} (0, 0) =  s'_{11} (0, 0)  = 0$. Taking $\xi = 0$ dividing out by$x$ and letting $x$ tend 
 to $0$ implies  that $   s'_{11} (0, 0)  = - 2    +(1 +a^2 ) s'_{23}(0, 0)$
 and $s'_{33}(0,0) =  a  s_{2,3} (0, 0) $. 
 These conditions can be met only if $a  = 0$. 
  \end{proof}
  
  When $a(x) = x$, by Corollary~\ref{corholdsym}, there is a Lipschitz symmetrizer, but it turns out 
  that in this specific case, one can construct a $C^\infty$ symmetrizer. 
  The next example shows that this is not always the case. 
  
  \bigbreak
  
  \noindent{\bf Example 2 : }  Consider the $4 \times 4 $ system with symbol 
 \begin{equation}
 L_a   =  \D_t   +    \begin{pmatrix} \Omega  &  a  J   \\ 0   & 2 \Omega  \end{pmatrix}, 
 \qquad \Omega = \begin{pmatrix} \xi & x\eta \\ x\eta  & - \xi \end{pmatrix}, 
 \qquad  J  = \begin{pmatrix} x \eta  & 0\\ 0  & 0 \end{pmatrix}. 
 \end{equation}
 By Corollaray~\ref{corholdsym}, when $a(x) = x$, $L_a$ has a Lispchitz  symmetrizer but

 \begin{lem}
 When $a(x) = x$, there are no $C^1$ symmetrizers for $L_a$.  
 \end{lem}
 
 \begin{proof}
 Fix  $\eta = 1$. Suppose that $S (\xi,  x) $ is a  $C^1$ symmetrizer near $(x, \xi) = (0,0)$.  
  We can assume that $S$ has real coefficients. Using the block notation 
 \begin{equation}
S  = \begin{pmatrix} S_{11} &  S_{12}  \\  S_{21}   & S_{22}\end{pmatrix}, 
\end{equation}
the symmetry conditions imply  
 \begin{equation}
 \label{exple2b} 
 \Omega  S_{12}   -   2 S_{12} \Omega  = x^2 J_0 S_{11} , 
 \qquad J_0 = \begin{pmatrix} 1  & 0\\ 0  & 0 \end{pmatrix}.
  \end{equation}
This is a linear system in $S_{12}$ and since 
$\Omega $ and $2 \Omega$ have no common eigenvalue it has a unique solution. 
 
 If $S_{12}$ is $C^1$ near the origin, plugging  its Taylor expansion
 $\Sigma_0 + x \Sigma_1 + \xi \Sigma_2$ in \eqref{exple2b}
 and using the notation   $\Omega = x \Omega_1 + \xi \Omega_2$, yields at first order
  \begin{equation*}
  \Omega_1 \Sigma_0 -  2 \Sigma_0 \Omega_1 = \Omega_2 \Sigma_0 -  2 \Sigma_0 \Omega_2 = 0 
\end{equation*}
  which implies that $\Sigma_0 = 0$. The term in $\xi^2$ is
   \begin{equation*}
\Omega_2 \Sigma_2 -  2 \Sigma_2 \Omega_2 = 0 
\end{equation*}
  showing that $\Sigma_2 = 0$. The term in $x \xi$ is then 
  \begin{equation*}
  \Omega_2 \Sigma_1 -  2 \Sigma_1 \Omega_2 = 0 
\end{equation*}
  implying that $\Sigma_1 = 0$, which is incompatible with the equation given by
  the term in $x^2$: 
  \begin{equation*}
\Omega_1 \Sigma_1 -  2 \Sigma_1 \Omega_1 =  - 2 J_0 S_{11} (0, 0)  \ne 0
\end{equation*}
  since $S_{11}(0,0)$ must be definite positive. 
 \end{proof}



 \subsection{Existence of smooth symmetrizers for generic double eigenvalues}
 
 Consider a symbol $ \tau \Id  + A(a, \xi)  $  which is strongly hyperbolic in the time direction, thus admitting a bounded 
 symmetrizer $S(a, \xi)$. 
 At  $(\underline a, \underline \xi)$, $\underline \xi \ne 0$, the characteristic polynomial   $p(a, \tau, \xi) = \det (\tau \Id  + A(a, \xi) ) $
 has roots $\underline \tau_j$ of multiplicity $m_j$.  Near this point,  
 can be smoothly factored 
 \begin{equation}
\label{factor}
 p(a, \tau, \xi)  = \prod_j  p_j (a, \tau , \xi)
\end{equation}
with $p_j $ of order $m_j $. 
 
 \begin{ass} In a neighborhood of 
 $(\underline a, \underline \xi)$, $ \underline \xi \ne 0$, the roots of $p$  are either of constant multiplicity 
 or of multiplicity at most two. 
 
 In the second case, we assume that the multiplicity is two on a smooth manifold $\cM$. Denoting by 
 $p_j$ the corresponding factor in $\eqref{factor}$, we further assume that either  
 
 i)  $\cM$ has codimension one and the the discriminant of   $p_j$ vanishes on $\cM$ at finite order, 
 
 \noindent or
 
 ii) $\cM$ has codimension one and the the discriminant of   $p_j$ vanishes on $\cM$ exactly at order two. 
 
 \end{ass} 
 
 \begin{theo}
 \label{theodouble}
 Under these assumptions, there is a smooth symmetrizer $S(a, \xi)$ on a 
 a neighborhood of 
 $(\underline a, \underline \xi)$. 
 \end{theo}

 \begin{proof}
 The construction is local in $ \rho =  (a, \xi)$ and one can perform a block reduction 
 of $A$ near $\underline \rho$ and it is sufficient to construct a symmetrizer 
 for each. The blocks are either diagonal an thus symmetric, or of dimension  two. 
 Eliminating the trace, it is sufficient to consider matrices 
 \begin{equation}
A (\rho) =   \begin{pmatrix} - a & b \\ c & a \end{pmatrix}  . 
 \end{equation}
 The hyperbolicity condition is that  the discriminant 
 $\Delta =  a^2  + b c $ is real and non negative. 
 Strong  hyperbolic, holds if and only if  there is $\eps > 0$ such that 
 \begin{equation}
 \label{stronghyp22}
 \Delta = a^2  + b c \ge \eps (|a| ^2 + | b |^2 + |c|^2). 
 \end{equation}
 Our assumption is that $\Delta$ vanishes on a manifold $\cM$, 
 at finite order if  $\mathrm{codim\,} \cM = 1$ and at order two if $\mathrm{codim \,}  \cM = 2$.

 { \bf a) } If $\Delta $ vanishes at finite order on a manifold of codimension $1$ of equation $\{ \vp = 0\}$, then 
 \eqref{stronghyp22} implies that for some integer $k$, 
 \begin{equation}
 A = \vp^k  A_r 
 \end{equation}
 with $\det A_r \ne 0 $ and still strongly hyperbolic. Thus $A_r$ has distinct real eigenvalues and is 
 therefore smoothly diagonalizable. 
 
 \bigbreak
 
 {\bf b) } Suppose that $\Delta$ vanishes exactly at second order on $\cM$
 given by the equations $\{ \vp = \psi = 0\}$.  This means that 
 $\Delta \ge \eps_1 (\vp^2 + \psi^2)$. 
 Together whith \eqref{stronghyp22}, this  implies that $A$ vanishes 
 on $\cM$ and that  
 \begin{equation}
A = \vp  A_1  + \psi A_2
\end{equation}
and  $A_1$ and $A_2$ have distinct real eigenvalues at $\underline \rho$. We can smoothly conjugate $A_1$ 
to a real diagonal and traceless form and changing $\vp$ we are reduced to the case where 
$$
A_1 = \begin{pmatrix}
   -1   &0    \\
     0 &  1
\end{pmatrix}, \qquad A_2 = \begin{pmatrix}
   - a_2     & b_2     \\
    c_2  &    a_2
\end{pmatrix}.
$$
  Moreover, changing $\vp$ to $\vp - \re a_2 \psi$, we can assume that  $\re a_2 = 0$. 
  Since $\Delta $ is real 
 \begin{equation}
\label{realdelta}
  2 \vp    \im a_2  + \psi  \im (b_2 c_2 ) = 0 
\end{equation}
  Moroever,  \eqref{stronghyp22} implies for $\vp = 0$
 \begin{equation}
\label{positiv22}
  \re (b_2 c_2) >  (\im a_2)^2
\end{equation}
  and this remains true in a neighborhood of $\underline \rho$.  In particular 
  $b_2 (\underline \rho ) \ne 0$ and conjugating by a diagonal matrix with diagoanal entries
  $b_2 / | b_2 |$ and $1$ changes $b_2$ into $|b_2$, meaning that we can assume that 
  $b_2$ is real. Having performed this reductions, one easily checks using  \eqref{realdelta}  that 
  \begin{equation}
\begin{pmatrix}
     \re c_2  &   i \im a_2 \\
   - i \im a_2   &   b_2 
\end{pmatrix}
\end{equation}
  is a smooth symetrizer for $\vp A_1 + \psi A_2$, which is definite positive by \eqref{positiv22}.  
\end{proof}


\subsection{Ill posedness for non Lipschitz symmetrizers}

Consider the system \eqref{exple1} with $a= a(x) = | x|^\alpha$.   For $\eta$ large and $\beta   $ to be determined, we look for solutions of $L_a U = 0$ 
of the form
\begin{equation}
\label{illp1}
U(t, x, y) = e^{ i \beta \sqrt \eta t + i y \eta}  \begin{pmatrix}
      u(\sqrt \eta x)    \\ v(\sqrt \eta x) \\ w(\sqrt \eta x) 
\end{pmatrix} . 
\end{equation}
With $\eps = \eta^{-\alpha/2}$,  
the equation $LU = 0$ is equivalent to 
\begin{equation}
\label{illp2}
v (x) = \frac{i }{\beta} ( \D_x -  i \eps x a  ) u (x)  , \qquad  w 
=  - \frac{1 }{\beta} ( 1  +  \eps^2    a^2  ) u (x)
\end{equation}
and the scalar equation for the first component is 
\begin{equation*}
\Big( \beta^2   +  ( \D_x +  i \eps x a  )   ( \D_x -  i \eps x a (x) )  - x^2 (1+ \eps^2 a^2) \Big) u = 0 
\end{equation*}
that is, since $\D_x(xa) =  (\alpha+1) a$, 
\begin{equation}
\label{ill3}
\Big( \beta^2   +  \D_x ^2   - x^2  - i \eps (\alpha+1)  a \Big) u = 0 . 
\end{equation}
The example has been cooked up precisely to get an eigenvalue problem for 
a perturbation of the harmonic oscillator.

When $\alpha = 0$, $\eps = 1$ and  $u(x) =  e^{ - \mez x^2} $ is a  a solution when 
\begin{equation}
\label{ill4}
\beta^2 = i -1. 
\end{equation}
Choosing the root with negative imaginary part, this yieds exact solutions of $L_a U = 0$ of the form 
\begin{equation}
\label{illp5}
U_\lambda (t, x, y) =  \int   e^{ i \beta \sqrt \eta t + i y \eta}  e^{- \mez \eta x^2}
( U_0 + \sqrt \eta x U_1)  \vp (\eta / \lambda) d \eta 
\end{equation}
with constant vectors $U_0 $ and $U_1$ not equal to $0$ and 
$\vp \in C^\infty_0 (\RR)$ with support in the interval $[1, 2]$. 
 The exponential growth of 
 $e^{ i t \sqrt \eta \beta}$ implies that there is no control 
 of any $H^{-s}$ norm at positive time by an $H^{s'}$ norm of the initial data. 
 This can be localized in $(x, y)$ and
 
 \begin{prop}
When $a = 1$, $L_a$ has bounded symmetrizers but the Cauchy problem for \eqref{exple1} is ill posed in $L^2$ but also in 
$C^\infty$
 \end{prop}

Consider now the case $\alpha \in ]0, 1[$. By standard pertubation theory the eigenvalue problem 
\eqref{ill3}
\begin{equation}
\label{ill5}
u = e^{ - \mez x^2} + \eps u_1 , \qquad \beta^2 = 1 + i  \eps \lambda_1 + O(\eps^2)
\end{equation}
with 
\begin{equation}
\label{ill6}
\lambda_1 \int  e^{ -  x^2} dx =   (\alpha+1) \int  a(x) e^{ -   x^2} dx  > 0 
\end{equation}
so that $\lambda_1 > 0$. Therefore one can choose 
$
\beta = - 1 - \frac{i}{2} \eps \lambda_1 + O(\eps^2)
$
\begin{equation}
\im (\beta \sqrt \eta) \sim \mez \lambda_1  \eta^{\mez (1- \alpha) } < 0
\end{equation}
which is arbitrarily large  if $\alpha < 1$.  This provides solutions of $L_a U = 0$, 
with exponentially amplified   $L^2$ norms implying the following proposition.

 \begin{prop}
When $a = | x|^\alpha$ with $0< \alpha < 1$, $L_a$ has $C^\alpha$  symmetrizers 
but the Cauchy problem for \eqref{exple1} is ill posed in $L^2$. 
 \end{prop}



\section{Strong hyperbolicity of first order symbols}
 \label{sectionstroghyp}

 In this section we introduce the notion of strong hyperbolicity and show that is is equivalent 
 to the existence of symmetrizers. Next we discuss the existence of smooth 
 symmetrizers. We show that these notions are preserved by a change of  the time direction. 
For the convenience of the reader, we postpone to the appendix  
the proof of several independent results on matrices.

 \subsection{Basic properties} 
We denote by $\tilde x  \in \RR^{1+d}$ the time-space variables and by 
 $\tilde \xi  $ the dual variables.  
We consider  $N \times N$ first order system systems $  \sum_{j=0}^d  A_j \D_{\tilde x_j} + B $. 
Their characteristic determinant is $p(\tilde \xi) = \det\big(   \sum_{j=0}^d  i  {\tilde \xi_j}A_j + B\big) $, 
the principal part of which is $p_N (\tilde \xi) = \det\big(   \sum_{j=0}^d  i  {\tilde \xi_j}A_j \big) $

\begin{defi}

i)    $  \sum_{j=0}^d  A_j \D_{\tilde x_j} + B $  is said to be hyperbolic 
 in the direction $\nu  \in \RR^{1+d} $ if the characteristic determinant 
   is, that is  if 
$p_N (\nu) \ne 0$ and there is $\gamma_0$ such that  
 $p(i \tau \nu + \tilde \xi) \ne  0 $ for all $\xi \in \RR^{1+d}$ and all real   $| \tau |   >  \gamma_0$. 
 
ii)  $L =  \sum_{j=0}^d  A_j \D_{\tilde x_j} $ is strongly hyperbolic  in the direction $\nu$ if and only if for all matrix $B$,  $L+B$ is hyperbolic in the direction $\nu$.
 \end{defi}

The classical definition of hyperbolicity is that the roots of $p(i \tau \nu + \xi) \ne  0 $ are located in $\tau   <   \gamma_0$.  But, since hyperbolicity  in the direction $\nu$ implies hyperbolicity   in the direction $- \nu$, the definition above is equivalent to the usual one.

\begin{prop} 
\label{invhyp1}
 $L =  \sum_{j=0}^d  A_j \D_{\tilde x_j} $ is strongly hyperbolic  in the direction $\nu$ if and only if  
there is a constant $C $ such that 

\quad i) for all $\tilde \xi\in \RR^{1+d} $ and all 
matrix $B$, the roots  of $\det \big( L(\tilde \xi + \lambda \nu) + B\big) = 0 $ are located in the strip
$| \im \lambda | \le C | B|$,  

  With the same constant $C$, this condition is equivalent to 

\quad ii)   for all $(\gamma, \tilde \xi, u) \in \RR \times \RR^{1+d} \times \CC^N$: 
\begin{equation}
\label{resolvp1}
\big|  \gamma   u \big| \le C \big| L(\tilde \xi + i \gamma \nu) u \big| .  
\end{equation}

\end{prop}  

Other equivalent formulations can be deduced from Proposition~\ref{propstronghyp}  below.

 \begin{proof}
{\bf a) }
 By homogeneity,   $ii)$  is equivalent to the condition 
 \begin{equation}
  \label{eqC2prime}
  \qquad  \big| \im \lambda | \ge C \quad \Rightarrow \quad
 \big|   L(\tilde \xi +  \lambda \nu )^{-1} \big| \le   1 .       
 \end{equation} 
By Lemma~\ref{leminv} below, this is equivalent to the condition 
that for all matrix $B$ such that $| B| < 1$,  $L(\tilde \xi +  \lambda \nu ) + B$ is invertible   
when $| \im \lambda | \ge C $. This is equivalent to saying  that the roots of  
$\det\big(L(\tilde \xi + \lambda \nu)  + B \big)  = 0$ are contained in $\{ | \im \lambda | < C\} $. 
By homogeneity, this is equivalent to $i)$.  

{\bf b) }    Note that \eqref{resolvp1}  applied to $\tilde \xi=0$ 
 implies that $L(\nu) $ is invertible. It is then clear that $i)$ implies 
 strong hyperbolicity.  Conversely, assume that $L$ is strongly hyperbolic. 
 Consider the matrix $ B_{j, k}$ with all entries equal to zero, except the entry 
of indices $(j, k)$ equal to one. Then 
\begin{equation*}
\det \Big( L(\tilde  \xi ) + B_{j,k} \Big) = 
\det  L(\tilde  \xi )  +  m_{j k}(\xi )
\end{equation*} 
where $m_{j,k}$ is the cofactor of indices $(j, k)$ in the matrix $L(\tilde  \xi )$. 
Following   Theorem 12.4.6 in \cite{Ho1}, the hyperbolicity condition 
implies that   there is a constant $C$ such that 
  \begin{equation*}
\label{estimqp} 
\big|   m_{j, k}  (\tilde \xi + i   \nu) \big| \le  C    \big| \det L (\tilde \xi + i  \nu) \big|. 
\end{equation*}
Since 
$L(\tilde \xi + i \nu)^{-1}  = (\det L(\tilde \xi + i \nu))^{-1} \widetilde M(\tilde \xi + i \nu)$ where 
$\widetilde M$ is the matrix with entries $(-1)^{j+k} m_{k,j}$, 
this implies that there is another constant $C$ such that \eqref{resolvp1} is satisfied for $\gamma =1$. By homogeneity, it is also satisfied for all 
$ \gamma$.
   \end{proof}

 When  $p$ is hyperbolic in the direction $\nu$, then  the component of 
$\nu $ in the set $\{ p _N   \ne 0 \}$ is a convex open cone, which we denote by   $\Gamma (\nu) $  
and $p$ is hyperbolic in any direction $\nu' \in \Gamma (\nu)$. 
This property is also true for strong hyperbolicity. We give a quantitative 
version of this result, as we will need is later on. 

\begin{lem}
\label{lemopencone}
Suppose that $L$ is hyperbolic in the direction $\nu$.  For all  $\nu'\in \Gamma(\nu)$, 
the ball centered at $\nu'$ of radius $\eps : =     |p_N( \nu') | / K |\nu'|^{N-1}$  is contained in 
$\Gamma(\nu)$, where 
$ K = \max_{| \tilde \xi | \le 2}  | \na_{\tilde \xi}  \det L(\tilde \xi) | $.
\end{lem}

\begin{proof}
By homogeneity, one can assume that 
$|\nu' | = 1$. In this case,  
$|p_N(\nu '')| >  | p_N(\nu')|  - K | \nu'- \nu"|  $ if $\nu" - \nu'| \le 1$. 
Noticing that $| p_N (\nu') | = | \nu' \na_{\tilde \xi} p_N(\nu') | \le K$, this implies that 
 $|p_N(\nu '')| > 0  $ if $|\nu" - \nu'| \le \eps $. 
\end{proof}

\begin{prop}
\label{propinvnu}
Suppose that $L(\tilde \xi ) $ satisfies \eqref{resolvp1} and let $\nu' \in \Gamma(\nu) $  such that 
$\big| \det L(\nu') \big| \ge c  > 0$.  Then 
\begin{equation}
\label{resolvp2}
\big|  \gamma   u \big| \le C_1 \big| L(\tilde \xi + i \gamma \nu') u \big| .  
\end{equation}
with $C_1 =  K  C | \nu| /  c | \nu'| $  and 
$ K = \max_{| \tilde \xi | \le 2}  | \na_{\tilde \xi}  \det L(\tilde \xi) | $
 \end{prop} 

\begin{proof}
Let $p (\tilde \xi) = \det ( L(\tilde \xi) + B)$ and 
$p_N = \det L(\xi)$ its principal part. 
Suppose that $|\nu | = | \nu' | = 1$. The general case follows immediately. 
By Proposition~\eqref{invhyp1}, one has 
\begin{equation}
\label{eq501}
p(\tilde \xi +  i \gamma \nu) \ne 0  \quad \mathrm{for} \   | \gamma | > C | B| . 
\end{equation}

We choose $\nu'' =  \nu' - \eps \nu $  with $\eps =  c / K$. 
 By Lemma~\ref{lemopencone}  $\nu'' \in \Gamma(\nu)$ and  following   
 \cite{Gard1, Ho1}, (see  e.g. \cite{Ho1}  vol 2, chap 12),  one has
 \begin{equation}
\label{eq51}
p( \tilde \xi + i \gamma \nu + i \sigma \nu'') \ne 0  \quad \mathrm{for} \    \gamma   > C ,  \   \sigma \ge 0.  
\end{equation}
and also for $  \gamma < -  C |B| $ and $\sigma \le 0$.  
Indeed,   all the roots of  $ p_N(t \nu + \nu'')$ are real and negative: 
\begin{equation}
\label{eq50}
  p_N (t\nu + \nu'') = 0
\quad \Rightarrow  \quad  t < 0. 
\end{equation}
By \eqref{eq501},   
$p(\tilde \xi + i \gamma \nu + z\nu'') = 0$  has no root on the real axis, so that 
the number of roots in $\{ \im z \ge  0 \} $ is independent of 
$\tilde \xi$ and $\gamma > C | B| $. Taking $\tilde \xi = 0$ and letting  $\gamma  $ tend to 
$+ \infty$,   \eqref{eq50}  this implies that this number is equal to zero implying \eqref{eq51}. 
The proof for $\gamma \le - C | B|$ and $\sigma \le 0$ is similar. 

 Substituting $\nu'' = \nu' - \eps \nu$ in \eqref{eq51} and choosing $\gamma = \eps' \sigma$ 
 we conclude that 
 $$
 p(\tilde \xi + i \sigma \nu' ) \ne 0 
 $$
  if $  \eps | \sigma |  > C | B| $.  Applying again  Proposition~\ref{invhyp1},  
  \eqref{resolvp2} follows. 
 \end{proof}

In most applications, the coefficients of the system $L$ and even the direction $\nu$ may depend 
on parameters, such as the space time variables, the unknown itself etc. The direction $\nu$ itself 
can be seen as a parameter.  
This leads to consider families of systems, $L(a, \D_{\tilde x})$ and directions $\nu_a$,  depending on 
 parameters $a \in \cA$. Their symbol is 
  \begin{equation}
 \label{gesymb}
 L (a, {\tilde \xi})   = \sum_{j = 0}^d  \tilde \xi_j A_j  (a)  . 
 \end{equation} 
 When considering such families, we always assume that 
the matrices $A_j(a)$  and the directions $\nu_a$ are  uniformly bounded,

 \begin{defi} 
We say that the family $ L(a, \, \cdot\, )$   is uniformly strongly hyperbolic  in the direction $\nu_a$ for 
$a \in \cA$ if ,

i)  $c_\cA:=  \inf_{a \in \cA}|  \det L(a, \nu_a) |  > 0$, 

ii) the equivalent conditions $i)$ and  $ii) $ of Proposition~$\ref{invhyp1}$  are satisfied with a constant 
$C$ independent of $a \in \cA$.   
 
  \end{defi}
 
The next result is a immediate consequence of  Proposition~\ref{propinvnu}. 
It shows that  one can enlarge the set of strongly hyperbolic direction, preserving uniformity: 
let  $\Gamma_a$ denote the component of $\nu_a$ in $\{\det L(a, \tilde \xi) \ne 0\}$; for  $c \in ] 0, c_{\cA}] $  and $C >0$ 
introduce the set 
\begin{equation}
\label{tildecA}
\widetilde \cA = \big\{ (a, \nu) ;   a \in \cA,  \nu \in \Gamma_a,  | \nu | \le C   ,  | \det L(a, \nu ) |  \ge c
\big\}. 
\end{equation}

\begin{prop} 
\label{prophyppol}
Suppose that  $ L(a, \, \cdot\, )$   is uniformly strongly hyperbolic  in the direction $\nu_a$ for 
$a \in \cA$.  
Then  $ L(a, \, \cdot\, )$   is uniformly strongly hyperbolic  in the direction $\nu$, for 
$(a, \nu) \in \widetilde \cA$. 
  
\end{prop}


\subsection{Symmetrizers} 

We start with the notion of \emph{full symmetrizer} introduced in \cite{FriLa1}.    

\begin{defi}
A full  symmetrizer for $L(\tilde \xi) = \sum \tilde \xi_j A_j$   is a bounded 
matrix ${\bf S} (\tilde \xi) $, homogeneous of degree $0$ 
on $\RR^{1+d} \backslash\{0\}$, such that 
$\bS(\tilde \xi) L(\tilde \xi)$ is self adjoint. 
It is positive in the direction $\nu \ne 0$ if there is a constant $c> 0$ such that 
for all $\tilde \xi \ne 0$ : 
\begin{equation}
\label{posfullsym}
u \in \ker L(\tilde \xi) \quad \Rightarrow \quad 
\re \big( \bS(\tilde \xi) L( \nu ) u, u \big) \ge c | u|^2 . 
\end{equation}
Given a family of systems $L(a, \, \cdot)$ and directions $\nu_a$, a bounded family 
of full symmetrizers  $\bS(a, \, \cdot \,) $ for $a\in \cA$ is said to be uniformly positive 
in the direction $\nu_a$ if the constant $c$ above can be chosen independent of $a$. 
\end{defi} 

In \eqref{posfullsym}, $ (f, u ) $ denotes the hermitian scalar product in $\CC^N$. 
More intrinsically, it should be thought as the antiduality between covectors $f$ and vectors $u \in \CC^N$,  
 so that the adjoint $P^*$ of a the matrix $P$ satisfies
 $(f, Pu) = (P^*f, u)$.

 Note that away from the 
the characteristic variety, the symmetrizer can be chosen arbitrarily and thus 
contains no information. 

\bigbreak

The may be more familiar notion of symmetrizer depends on the choice 
of a time direction $\nu$. Choosing  a space $\EE $ such that $ \RR^{1+ d} = \EE \oplus \RR \nu$
the symmetrizer is seen as a function of frequencies $\xi \in \EE$.  
Since the open cone  $\Gamma(\nu)$ is strictly convex, one can also require that 
$\Gamma(\nu) \cap \EE = \emptyset$. 
 In a more intrinsic  definition, it  can be seen as a symmetrizer 
invariant by translation in the direction $\nu$, or defined on $\RR^{1/d} / \RR \nu$.
To avoid technicalities, we choose the first option choosing a space $\EE$. 
and when considering families $(L(a), \nu_a)$, we assume that we can choose  $\EE$  
in such a way  that there is a compact set $K$ such that 
\begin{equation}
\label{unifdecomp}
\forall a \in \cA,  \nu_a \in K \qquad \mathrm{and} \qquad    K \cap \EE =  \emptyset . 
\end{equation}
This condition can always be met locally. In particular, uniformly in $a \in \cA$: 
\begin{equation}
\label{unifdecomp2}
c  (| \xi | + | \tau|)  \le  | \xi + \tau \nu_a |   \le C  ( | \xi | + | \tau|), \qquad \xi \in \EE, \ \tau \in \RR. 
\end{equation}

\begin{defi}
A   symmetrizer for $L(\tilde \xi) = \sum \tilde \xi_j A_j$  in the direction $\nu$  is a  bounded 
matrix $S( \xi) $, homogeneous of degree $0$ in  $\xi \in \EE$ such that  
  $S( \xi) L(  \xi)$  and $S( \xi) L(\nu)$ are
self adjoint  for all $  \xi $  and there is $c > 0$ such that : 
\begin{equation}
\label{possym}
\forall \xi \in \EE\backslash\{0\}, \ \forall  u  \in \CC^N \, , \qquad    \big( S(\xi) L( \nu ) u, u \big) \ge c | u|^2 . 
\end{equation}
Given a family of systems $L(a, \, \cdot)$ and directions $\nu_a$ satisfying \eqref{unifdecomp}, a uniform   family 
of  symmetrizers  $S(a, \, \cdot \,) $ for $\{L(a, \cdot), \nu_a\} $, $ a\in \cA$, is a bounded family $S(a, \cdot)$ of symmetrizers 
for $L(a, \cdot)$  in the direction$ \nu_a$, such that the constant $c$  can be chosen independent of $a$. 
\end{defi} 

\begin{rem}
\label{remdefsym}
\textup{ $\bS (\xi + \tau \nu) = S(\xi)$, is almost a full symmetrizer, except that it not necessarily 
defined   on the line $\RR \nu$. But  \eqref{possym} implies that 
$L(\nu ) $ is invertible and one  can always choose 
 $\bS(\nu)$ so that  
$\bS(\nu) L(\nu)$ is definite positive. 
This modification can be extended to a conical neighborhood of $\nu$
where $L (\tilde \xi))$ remains invertible.  This construction 
obviously preserves positivity. 
 This remains true for families and uniformity can be preserved}. 
\end{rem}

The existence of symmetrizers is equivalent to strong hyperbolicity, in the 
following sense.

\begin{theo}
\label{theostronghyp}
Consider a family $\{L(a, \, \cdot\,), \nu_a, a \in \cA\} $.

i) Assuming \eqref{unifdecomp}, $L(a,\cdot )$ is strongly hyperbolic in the direction $\nu_a$ 
if and only if  there exists a  uniform  family of symmetrizers   $S(a, \cdot) $.

ii)  $L(a,\cdot )$ is strongly hyperbolic in the direction $\nu_a$ if and only if 
  
\qquad a)  is hyperbolic in the direction $\nu_a$ and $\inf_{a \in \cA} | \det L(a, \nu_a) | > 0$

\qquad b) there is a   bounded  family of full symmetrizer $\bS(a, \cdot) $  which is 
uniformly  positive in the direction $\nu_a$.

\end{theo}

\begin{proof}
{\bf i) } If $L(a, \cdot) $ is uniformly strongly hyperbolic in the direction $\nu_a$, 
then $L(a, \nu_a)$ and $L(a, \nu_a)^{-1}$ are uniformly bounded, 
Similarly, \eqref{possym} implies that   $L(a, \nu_a)^{-1}$ is bounded. 

In both case,  $ A(a,  \xi) =  L(a, \nu_a)^{-1} L(a,  \xi)$ for $|  \xi | = 1$ is bounded, and 
strong hyperbolicity is equivalent to  the existence of a constant $C$ such that 
for all $\lambda$, $a$, $ \xi \in \EE $ and $u$: 
\begin{equation}
\label{resolvestp11}
| \im \lambda |  \big| u \big| \le C  \big| A(a,   \xi) u  - \lambda  u \big| . 
\end{equation}
By Proposition~\ref{propstronghyp} this is equivalent to the existence of
 a symmetric matrix 
 $S_A(a,  \xi)$, bounded and uniformly definite positive, such that $S_A A$ is symmetric. 
 This is equivalent to the condition that   $S (a,   \xi ) =  S_A(a,  \xi) L(a, \nu_a)^{-1}$  is a symmetrizer 
 for $L(a, \cdot) $ bounded and uniformly positive in the direction $\nu_a$.

 \medbreak
 
 {\bf ii) } Strong hyperbolicity implies the existence of a symmetrizer, thus of a full positive symmetrizer 
 by Remark~\ref{remdefsym}. Hence it only remains to prove the converse part of $ii)$.

Let $\bS(a, \cdot) $ be a full symmetrizer for $L(a, \cdot) $, positive in the direction 
$\nu_a$. Suppose in addition that $L(a, \cdot)$ is hyperbolic in this direction, 
so that $L(a, \nu)$ is invertible. Then,  Proposition~\ref{propPSP} implies that 
when $\ker L(a, \tilde \xi) \ne \{0\}$, $0$ is a
semi simple  eigenvalue of $ A(a, \tilde \xi) = L(a, \nu_a)^{-1} L a, \tilde \xi)$. Moreover,  the spectral projectors, 
that is the projectors on $\ker A = \ker L $ parallel to the range of $A$ , are uniformly bounded. 
 Applied to $\tilde \xi + \tau \nu$, this implies that all the real eigenvalues of  
 $A(a, \tilde \xi)$ are semi-simple and all the corresponding spectral projectors are uniformly 
 bounded. Since $L$ is hyperbolic, all the eigenvalues are real and with Proposition~\ref{propstronghyp}
 this implies that  \eqref{resolvestp11} 
 is satisfied.  
\end{proof}



\subsection{Smooth symmetrizers}

We now consider a family $\{L(a, \, \cdot\,), \nu_a, a \in \cA\} $ where $\cA$ is an open set of 
some space $\RR^m$. We assume that the condition \eqref{unifdecomp} is satisfied.

\begin{theo}
\label{theomaj} 
Suppose that the coefficients of $L(a, \cdot)$ are continuous, [resp. $W^{1, \infty}$]
[resp. $C^\infty$] on $\cA$ and that the mapping $a \mapsto \nu_a$ is 
continuous, [resp. $W^{1, \infty}$]
[resp. $C^\infty$].  Then, there exists a  full symmetrizer $\bS(a, \tilde \xi)$ 
which is  continuous, [resp. $W^{1, \infty}$]
[resp. $C^\infty$] on $\cA \times S^d$ and uniformly positive in the direction $\nu_a$, 
if and only if there is a symmetrizer $S(a, \xi)$ 
which is  continuous, [resp. $W^{1, \infty}$]
[resp. $C^\infty$] on $\cA \times S^{d-1}$ and uniformly positive in the direction $\nu_a$, 
\end{theo}

\begin{proof}
By Remark~\ref{remdefsym} passing from a symmerizer to a full symmetrizer 
is immediate. The converse statement follows from a more general 
result given in Theorem~\ref{theomajext} where the construction 
is extended to other directions $\nu \in \Gamma(\nu_a)$. 
\end{proof}

 
 \subsection{Invariance by change of time}
 
Proposition~\ref{prophyppol} shows that strong hyperbolicity, thus the existence of bounded symmetrizers or 
of full symmetrizers, extends from $\nu_a$ to all directions in the  cone of hyperbolicity $\Gamma_a$, 
preserving uniformity in sets such as \eqref{tildecA}. 
We now prove that this is also true for smooth symmetrizers. 
The key point, is to prove that for a continuous full symmetrizer, 
positivity extends from $\nu_a$ to $\Gamma (\nu_a)$.

\begin{prop}
\label{propextpos}
Consider a family $\{L(a, \, \cdot\,), \nu_a, a \in \cA\} $ and assume that 
 $L(a, \cdot)$ is uniformly strongly hyperbolic in the direction $\nu_a$. 
For $c>0$ and $C$ given, let $\widetilde \cA$ as in \eqref{tildecA}. 

 Suppose that $\bS(a, \cdot )$ is a full symmetrizer of $L(a, \cdot)$ which depends continuously 
on $ \tilde \xi  \in \RR^{1+d} \backslash\{0\}$, such that  that $\bS(a, \cdot) $ is uniformly positive in the direction 
$\nu_a$. Then, $  \bS(a, \cdot) $  is uniformly positive in the direction $\nu$ for 
$(a, \nu) \in \widetilde \cA$. 
\end{prop} 

\begin{proof}
The symmetry of  $\bS(a, \tilde \xi + s \eta) L(a, \tilde \xi + s\tilde \eta)$ applied 
to $u$ and $v$ in $\ker L(a, \tilde \xi)$ implies that 
\begin{equation*}
\big( \bS(a, \tilde \xi + s \eta) L(a,\tilde \eta) u, v \big) 
= \big(u,  \bS(a, \tilde \xi + s \eta) L(a,\tilde \eta) v  \big) .
\end{equation*}
Letting $s$ tend to $0$, shows that for all $\eta$, 
the matrices  $\bS(a, \tilde \xi ) L(a,\tilde \eta) $ are symmetric on $\ker L(a, \tilde \xi)$. 

By Lemma~\ref{lemopencone}, there is $\eps$ such that for all $(a, \nu) \in \cP$, 
the ball centered at $\nu$ and radius $\eps$ is contained in $\Gamma_a(\nu_a)$. 
Therefore, there is $t_0 \in ]0, 1[$ such that for all $(a, \nu)\in \cP$, there is 
$\nu' \in \Gamma_a(\nu_a)$ on the line joining $\nu_a$ and $\nu$ such that 
$\nu = t \nu_a + (1-t) \nu'$ with $t \in [t_0, 1[$. 
By Proposition~\ref{propinvnu}, $L(a, \dot) $ is strongly hyperbolic
in the direction $\nu'$, implying that $L(a, \nu')^{-1} L(a, \tilde \xi)$ has 
only real an semi-simple eigenvalues. 
Therefore, the result follows from Proposition~\ref{positivcone} applied to
$J_t =  (1-t) L(a, \nu_a) + t  L(a, \nu) $
\end{proof}

We are now ready that the existence of a regular full symmetrizer
implies the existence of  symmetrizer, having the same smoothness, 
in all directions $\nu\in \Gamma (\nu_a)$. In particular, this finishes the proof of 
 Theorem~\ref{theomaj}. 
 Consider a strongly hyperbolic family $\{L(a, \, \cdot\,), \nu_a, a \in \cA\} $.  Assume that the condition \eqref{unifdecomp} holds. 
For  $c > 0$  and $C >0$ and $\cO$ an open neighborhood of $K$ such that $\overline \cO \cap \EE = \emptyset$
let 
\begin{equation}
\label{tildecA0}
\widetilde \cA_0 = \big\{ (a, \nu) ;   a \in \cA,  \nu \in \Gamma_a \cap \cO,  | \nu | \le C   ,  | \det L(a, \nu ) |  \ge c
\big\}. 
\end{equation}
 
 \begin{theo}
\label{theomajext}
 
 Suppose that the coefficients of $L(a, \cdot)$ are continuous, [resp. $W^{1, \infty}$]
[resp. $C^\infty$] on $\cA$ and that the mapping $a \mapsto \nu_a$ is 
continuous, [resp. $W^{1, \infty}$]
[resp. $C^\infty$].  Suppose that $S(a, \xi)$ is a uniform bounded family of symmetrizers  for $\{L(a, \, \cdot\,)$  
in the directions $\nu_a$ for $a\in \cA$, which is  continuous, [resp. $W^{1, \infty}$]
[resp. $C^\infty$] on $\cA \times \EE$. Then, there exist   a continuous, [resp. $W^{1, \infty}$]
[resp. $C^\infty$]  uniform family of symmetrizers   $S (a, \nu, \xi) $  for $L(a, \cdot)$ in the direction 
$\nu$, which is continuous, [resp. $W^{1, \infty}$]
[resp. $C^\infty$] on $\widetilde \cA_0 \times \EE$.

 \end{theo} 
  
  \begin{proof}
  For $\tilde a = (a, \nu) \in \widetilde \cA_0$, consider 
 \begin{equation*}
\tilde L(\tilde a, \tau, \xi) = \tau L(a, \nu) + L(a, \xi ) = L (a, \xi + \tau \nu). 
\end{equation*}
Then $\tilde \bS (\tilde a, \tau, \xi) = \bS(a, \xi + \tau \nu)$  symmetrizes $\tilde L(\tilde a, \tau, \xi)$. 
By \eqref{unifdecomp2}, $\tilde L$ and $\tilde \bS$ are 
continuous, [resp. $W^{1, \infty}$]
[resp. $C^\infty$] functions on $\widetilde \cA_0 \times \RR \times S^{d-1} $. 
The positivity condition 
\begin{equation}
\re \big( \tilde \bS (\tilde a, \tau, \xi) L(a, \nu) u, u \big) \ge c | u|^2
\end{equation} 
on $\ker \tilde L(\tilde a, \tau, \xi) = \ker L (a, \xi + \tau \nu)$ follows from Proposition~\ref{propextpos}
and and the construction of a symmetrizer 
$S(a, \nu, \xi)$, with the same smoothness as $\tilde \bS$,   is given by Theorem~\ref{theomajmod}. 
  \end{proof}


 
\section{Appendix A : Strongly hyperbolic matrices} 
 \label{SecStrHypMat}
 
We collect here the various technical results on matrices which have been used 
in the previous section. Changing slightly the notations, for instance  including $\xi$ 
or  $\nu$ among the parameters, we consider a family of $N \times N$ matrices, $A(a)$ depending on parameters $ a \in \Omega$ 
 where $\Omega$ is an open subset of $\RR^n$. We denote by $\Sigma(a)$ the spectrum of 
 $A(a)$. 
 
 \subsection{Definition and properties} 
 
 \begin{prop}
 \label{propstronghyp}
 The following properties are equivalent 
 
 i)  There is a real  $C_1$ such that
  \begin{equation}
  \label{eqC1}
 \forall t \in \RR,  \forall a \in \Omega \ : \qquad  
 \big| e^{ i t A(a)  } \big| \le C_1 . 
 \end{equation}

 ii) All the the eigenvalues $\lambda$ of  $A(a)$  are real  and semi-simple and 
  there is a real $C_2$ such that all the 
 eigen-projectors $\Pi_\lambda (a)$  satisfy 
  \begin{equation}
  \label{eqC3}
  \forall   a \in \Omega \ : \qquad  
 \big| \Pi_\lambda (a)  \big| \le C_2 . 
 \end{equation}     
 
 iii) $A(a)- \lambda \Id $ is invertible when $\im \lambda \ne 0$ and 
 there is a real $C_3$ such that
  \begin{equation}
  \label{eqC2}
 \forall \lambda \notin \RR \   \forall  a \in \Omega \ : \qquad  
 \big|  \big(  A(a)- \lambda  \Id \big)^{-1} \big| \le C_3  \big| \im \lambda |^{-1}. 
 \end{equation} 
 
  iv) There are definite positive 
  matrices $S(a)$ and there are constants  
 $C_4$ and  $c_4 > 0$ such that for all $a \in \Omega$,  $S(a) A(a) $  is   symmetric, 
   and   
 \begin{equation}
 \big| S( a) \big| \le C_4 , \qquad S( a)   \ge c_4 \Id. 
 \end{equation}  
 
 v)  There is a real $C_5$  such that for all matrix $B$, all $a\in \Omega$ and all 
 $\rho  \in \RR$, the eigenvalues   of $\rho A(a) + B$ are located in 
 $\{ | \im \lambda | < C_5 | B | \}$. 
     
 \end{prop}

 \begin{proof}

 {\bf a) }  $ii)$ implies that  $A(a)$ has the spectral decomposition 
 $A = \sum \lambda_j \Pi_j $ with  real $\lambda_j $'s. 
 Thus \eqref{eqC3} implies that
 $\big| e^{ i t A) } \big| =  \big| \sum e^{ it  \lambda_j } \Pi_j\big| \le N C_2$.  
 
 Conversely,  $i)$ implies that the eigenvalues $\lambda_j$ of $A(a)$ are real, and semi-simple
 and thus  that $A(a) = \sum \lambda_j \Pi_j $. 
 Moreover,  
$$
 \lim_{ T \to \infty} \frac{1}{2T} \int_{-T}^T  e^{ i t (A(a) - \lambda_j \Id)}  dt  = 
 \sum_k   \lim_{ T \to \infty} \frac{1}{2T} \int_{-T}^T  e^{ i t (\lambda_k - \lambda_j \Id)} \Pi_k  dt   = \Pi_j . 
$$ 
Thus,  $| \Pi_j | \le C_1 $ if \eqref{eqC1}  is true.

\medbreak
{\bf b) }  Suppose that $ii)$ is satisfied so that $A = \sum \lambda_j \Pi_j$ and 
$\Id = \sum \Pi_j$. Then 
 \begin{equation}
 S  (a)  = \sum \Pi_j^* \Pi_j 
 \end{equation}
is definite positive, satisfies  $S  \ge N^{-1}  \Id$, $| S | \le N C^2_2$, and 
$S  A = \sum  \lambda_j \Pi_j^* \Pi_j  $ is self adjoint. 
 \medbreak

 If $ iv)$ holds then, with $\epsilon = \mathrm{sign}  (\gamma)$ , 
 $$
 \begin{aligned}
 c_4  |  \gamma |   \big| u\big|^2 \le \re \epsilon  \big( S  (- i  A + \gamma   \Id)   u, u \big) 
   \le  C_4 \big| (A + i \gamma)   u \big|  \ \big|  u \big|
 \end{aligned}
 $$
 implying $iii)$ with $C_3 = C_4/ c_4$.

 \medbreak
 If $iii)$  is satisfied,  then  the eigenvalues of $A(a)$ are real,  and semi-simple, for if there were a nondiagonal   block in the Jordan's decomposition of $A - \lambda_j \Id $ , the norm  of   
 $ (A -  (\lambda_j - i \gamma) \Id)^{-1}  $ would be at least of order $\gamma^{-2}$
when $\gamma \to 0$.  Thus $A = \sum \lambda_j \Pi_j$  and 
 $$
 \lim_{   \gamma  \to 0}  i \gamma  \big(A - ( \lambda_j - i \gamma) \Id)^{-1}  = 
 \sum_k  \lim_{  \gamma  \to 0} \frac{ i \gamma  }{(\lambda_k - \lambda_j + i  \gamma } \Pi_k = \Pi_j, 
$$
hence    $| \Pi_j | \le  \ C_3  $.

\medbreak
{\bf c) }  By homogeneity,   $iii)$  is equivalent to the condition 
 \begin{equation*}
 \forall  a \in \Omega, \forall \rho \in \RR,  \   \qquad  \big| \im \lambda | \ge C_3 \quad \Rightarrow \quad
 \big|  \big( \rho  A(a)- \lambda  \Id \big)^{-1} \big| \le   1 .       
 \end{equation*} 
By Lemma~\ref{leminv} below, this is equivalent to the condition 
that for all matrix $B$ such that $| B| < 1$, $\rho A - \lambda \Id + B$ is invertible f  
when $| \im \lambda | \ge C_3$, meaning that the spectrum of 
$\rho A + B$ is contained in $\{ | \im \lambda | < C_3$. 
By homogeneity, this is equivalent to $v)$ with $C_5 = C_3$. 

The proof of the proposition is now complete. 
   \end{proof}
  
  \begin{lem}
  \label{leminv}
  The matrix $A$ is invertible with $| A^{-1}| \le \kappa $ if and only if 
  $A+B$ is invertible for all $B$ such that $|B | < \kappa ^{-1} $. 
  \end{lem}
  
  \begin{proof}
  If $| A^{-1}| \le \kappa $, then $ A + B = A^{-1} (\Id + A^{-1} B)$ is invertible 
  for all $B$ such that  $| A^{-1} B| \le \kappa | B| < 1$. 
  
  Conversely, if $A$ is not invertible or if $| A^{-1}| > \kappa $, there is $\underline u$ such that 
  $| \underline u|  = 1$ and $| A \underline u | < \kappa^{-1} $. Pick a linear form 
  $\ell$ such that $\ell (\underline u ) = 1$ and $| \ell| = 1$. 
  Then the matrix $B$ defined by $B u = \ell(u) A \underline u$ 
  satisfies $| B | = | A \underline u| < \kappa^{-1} $ but $A-B$ is not invertible since 
  $\underline u$ is in its kernel. 
  \end{proof}


\subsection{Lipschitz dependence of the eigenvalues }

     \begin{ass}
     \label{assSHM}
   The  family  $\{ A(a), a \in \Omega\}$ of $N \times N$ matrices,
 is uniformly strongly hyperbolic in the sense that the equivalent properties of Proposition 
 $\ref{propstronghyp}$ are satisfied. 
   \end{ass} 
 
\begin{prop}
\label{proplipvp}  Suppose that $A(\cdot) \in W^{1, \infty} (\Omega)$ satisfies Assumption~$\ref{assSHM}$. 
Denote by $\lambda_j(a)$, $1 \le j \le N$,  the eigenvalues of $A(a)$, labelled in the increasing order
and repeated accordingly to their multiplicity.
Then, the functions  $\lambda_j$ belong to $ W^{1, \infty} (\Omega)$
\end {prop}

\begin{proof}
The continuity of the roots of a polynomial with respect to the coefficients is well known. 
The Lipschitz smoothness with respect to parameters of the roots of hyperbolic polynomials is true in general, provided that the coefficients are smooth enough (see \cite{Bron}). The proposition says that  in the case 
of the characteristic determinant of 
strongly hyperbolic systems, the Lipschitz smoothness of the coefficients is sufficient. 

 Fix $\underline a \in \Omega$ and an eigenvalue 
 $\underline \lambda= \lambda_p (\underline a) = \lambda_{p+ m} (\underline a) $ of $A(\underline a)$
 of multiplicity $m+1$. Let $\delta  >  0$ denote the distance of $\underline \lambda$ to the remainder part 
 of the spectrum of $A(\underline a)$. By   Assumption~\ref{assSHM}, there
 is $C$ which depends only on an upper bound of the norms of the spectral projectors, thus independent of 
 $\underline a$, such that 
\begin{equation*}
\forall z \in \CC  , \ 
| z - \underline \lambda | \le \delta/2 , \qquad  \big| \big(A(\underline a ) - z \Id \big)^{-1} \big| \le C | z - \underline \lambda | ^{-1} 
 \end{equation*}
  Therefore,  $A - z \Id$ is invertible when  $    | A - A(\underline a ) | <      | z - \underline \lambda |  \le \delta/2C$.

 Let $\Omega_1 \subset \Omega $ denote a convex open neighborhhood of $\underline a$. Because $A \in W^{1, \infty} (\Omega)$
for $a$ and $a' \in \Omega_1$ there holds 
 \begin{equation}
 \label{LipestA}
   \big| A(a) - A(a') \big| \le K | a - a '|
 \end{equation} 
 with $K = \big\| \nabla_a A \big\|_{L^\infty(\Omega)}$.  
 Therefore,   
 $ A(a) - z \Id$ is invertible if $a \in \Omega_1$ and  
 $ C K | a - \underline a  |  \le | z - \underline \lambda | \le \delta/2$. 
 By Rouché's theorem,  this implies that 
 $A(a)$ has $m$ eigenvalues (counted with their multiplicity) in the disk 
$\{ | \lambda - \underline \lambda | \le K C | a - \underline a|\}$. They must be real by assumption, and by continuity 
they are $\{ \lambda_p(a) , \ldots, \lambda_{p+m} (a) \}$. Hence, 
$ | \lambda_j(a) - \underline \lambda | \le  K C | a - \underline a|$ for $p\le j\le p+m$, 
 provided that $a \in \Omega_1$ and 
$  K C | a - \underline a | < \delta /2$. 

 Gluing these estimates together, we have proved that  there are constants $C$ and $K$  such that : 
for all  $\underline a \in \Omega$,  there is a convex neighborhood $\omega$  of $\underline a$, such that 
for $a \in \omega $ and all $j$,
$$
 | \lambda_j (a)  -  \lambda_j (\underline a)  | \le  C K | a - \underline a|. 
 $$
The following independent  implies that 
$ \big| \nabla_a \lambda_j  \big|_{L^\infty(\Omega)} \le CK$ and the proposition follows. 
\end{proof}

\begin{lem}
\label{lemLip}
Suppose that  $K $ is positive real number  and   $f$ is a function defined on the open set $\Omega \subset \RR^n$ such that 
for all  $  \underline a \in \Omega$,  there is a neighborhood $\omega$  of $ \underline a$, such that 
\begin{equation} 
\label{youpi} 
 \forall  a   \in \omega, \qquad  | f(a )  -  f( \underline a)  | \le K | a - \underline a|. 
\end{equation}
Then   $ \big\| \nabla_a f  \big\|_{L^\infty(\Omega)} \le K$. 
\end{lem} 
\begin{proof}
Note that \eqref{youpi} implies that $f$ is continuous at $\underline a$, thus $f$ is continuous. 

We first shown that for all convex open set 
$\Omega_1 \subset \Omega$ 
the inequality 
 \begin{equation}
\label{youpipi}
| f(b )  -  f(a)  | \le K   | b - a|. 
\end{equation} 
is satisfied  for all $a$ and $b$ in $\Omega_1$.
Indeedn let $\sT$ denote the set of real numbers 
$t \in [0, 1]$ such that 
\begin{equation} 
\label{youpi1} 
 \forall  s  \in [0, t] , \qquad  | f(a + s (b- a) )  -  f(a)  | \le K  s | b - a|. 
\end{equation}
By assumption, the property \eqref{youpipi} is satisfied on a 
neighborhood of $a$, implying  that $\sT$ is not empty. 
By definition $\sT$  is an interval, and by continuity of 
$f$ it is closed. Using the assumption \eqref{youpi}  near  $a +t (b-a)$, implies that 
$\sT$ is open so that $T = 1$ and \eqref{youpipi} is proved. 

This implies that $f$ is Lipschitz continuous on $\Omega_1$ and that
$ \big\| \nabla_a f \big\|_{L^\infty(\Omega_1)} \le K$. Since this is true for all ball $\Omega_1 \subset \Omega$, 
the lemma follows. 
\end{proof} 


\subsection{Lipschitz dependence of the  eigenprojectors}

\begin{prop}
\label{proplipproj}
Suppose that $A(\cdot) \in W^{1, \infty} (\Omega)$ satisfies Assumption~$\ref{assSHM}$. 
Let $\underline a \in \Omega$ and  consider 
$ \underline \Lambda := \{ \lambda_j (\underline a) , j \in J \}$ a subset of the spectrum 
of $A(\underline a)$. 
Let $ \underline \Lambda ' = \Sigma(\underline a) \backslash \underline \Lambda 
= \{ \lambda_j (\underline a) , j \in J' \}$  and define 
\begin{equation}
\delta = \mathrm{dist} (\underline \Lambda, \underline \Lambda') = 
\min_{ (j,j')  \in J,\times J'}  \big| \lambda_j (\underline a) - \lambda_{j'} (\underline a) \big| > 0 
\end{equation}
Let $\omega$ be a neighborhood of  $\underline a $ 
such that $|\lambda_j (a) - \lambda_j(\underline a) | \le \delta/4$ for all $j\in \{1, \ldots, N\}$ 
and $a \in \omega$. With 
$ \Lambda (a)  = \{ \lambda_j ( a) , j \in J \}$, 
consider the
  spectral projector 
\begin{equation}
\Pi_\Lambda (a) = \sum_ {\lambda \in \Lambda (a)}  \Pi_\lambda (a) . 
\end{equation} 
Then, 
$\Pi_\Lambda (a) $ is continuous on $\omega$ and 
\begin{equation}
\label{estlipproj} 
\big|   \Pi_\Lambda  (a )  - \Pi_\Lambda (  a') \big|  \le C K   \delta^{-1} \big| a -  a'\big|, 
\end{equation}
where    $C$   depends only on an upper bound of the norms of the spectral projectors of $A(\cdot)$ 
and $K = \big\| \nabla_a A \big\|_{L^\infty(\Omega)}$.  
\end{prop} 

\begin{proof} There are finitely many Jordan curves $\Gamma_k$ in the complex domain, of total length less than  $C \delta$, such that $|z - \lambda_j(\underline a) | \ge \delta /2$ for all 
$z \in \cup \Gamma_k $ and all $j \in J $,  surrounding 
$\Lambda$ so that   
$$
\Pi_\Lambda ( \underline  a) = \sum_k  \frac{1}{2i \pi} \int_{\Gamma_k}    (z \Id - A( \underline a) )^{-1} dz  
 $$
This formula extends to $a \in \omega$. 
Moreover,  using the estimate 
 $$
 \begin{aligned}
 \big|  (z \Id - A(   a) )^{-1}  -    (z \Id - A(  a') )^{-1} \big| \qquad \qquad &
 \\
\le   \big|  (z \Id - A(   a) )^{-1}  \big| \, \big| A(a)   - A(a') \big|  \, \big|   (z \Id & - A(  a') )^{-1} \big|  
\\
&   \le C K  \delta^{-2}  \big| a - a' \big|
 \end{aligned}
 $$
for $z \in \cup \Gamma_k$, implies \eqref{estlipproj}.
\end{proof}


\subsection{A piece of functional calculus}

We study the smoothness of  $f(A(a))$, given the  smoothness of $f$ and $A$.  
We extend the analysis to vector or matrix valued functions $S(\lambda, a)$ using the following definition
\begin{equation}
\label{functcalc}
S_A(a)  = \sum_{\lambda \in \Sigma(a)}  S(\lambda, a ) \Pi _\lambda (a ). 
\end{equation}

\begin{theo}
\label{theoFC}
Suppose that $A(\cdot)$ is continuous [resp. $W^{1, \infty}$] [resp. $C^\infty$]  on 
$\Omega$ and satisfies Assumption~$\ref{assSHM}$. Suppose that $S(\lambda, a)$ is  continuous [resp. $W^{1, \infty}$] [resp. $C^\infty$ ] 
on $\RR \times \Omega$. Then  
 $S_A(a)$ 
is continuous [resp. $W^{1, \infty}$] [resp. $C^\infty$ ] on $\Omega$. 
\end{theo} 

\begin{proof}[{\bf 1) } The $C^\infty$ case]
If $S$ were holomorphic in $\lambda$ one would have
$$
S_A(a) = \frac{1}{2i \pi} \int_{\partial D}  S (z,   a) (z \Id - A( a) )^{-1} dz  
$$
where $D $ is a rectangle  $[-R, R] + i [-\delta, \delta ]$ containing $\Sigma (a) $ in its interior, 
implying the result since $(z \Id - A(  a) )^{-1}$ is smooth in $a$ for $z \in \D D$. 
In the $C^\infty$ case, we modify this proof considering 
an almost holomorphic extension of $S$ in the variable $\lambda$.  It is a $C^\infty$ function in 
    $(z,  a)  \in \CC\times \Omega $ such that, for $z$ in bounded sets, 
    \begin{equation}
 \D_{\overline z } S ( z,   a)  = O(| \im  z |^\infty)
    \end{equation} 
 Since the result is local, we can assume that $\Omega $ is bounded and fix $R$ such that for 
 $  a \in \Omega$, the spectrum of  $  A(a)$ is contained in  
 $\{| z | \le R \}$.  Let $D$ denote the disc of radius $R+1$ in $\CC$. Then, 
 \begin{equation*}
 \begin{aligned}
   S_A  (a) = \frac{1}{2i \pi} & \int_{\partial D}  S (z,   a) (z \Id - A( a) )^{-1} dz  
\\
&    + \frac{1}{2i \pi}  \int_D \D_{\overline z}  S (z,   a) (z \Id - A( a) )^{-1} dz   d \overline z 
\end{aligned}
 \end{equation*} 
The first integral is $C^\infty$ in $a$ as explained above. 
In the second, we note that 
for all $m$, $ \D_{\overline z} S (z, a)  =  (\im z )^m R_m (z,  a)$ with $R_m$ smooth and 
$(\im z)^m (z \Id - A(a) )^{-1}$ has $m-1$ uniformly bounded  derivatives in $a$,  for $z \in D \backslash \RR$,  
implying that the second integral is $C^{m-1}$ with respect to $a \in \Omega$. 
\end{proof}

\begin{proof}[{\bf 2)} Continuity] Fix $\underline a \in \Omega$ and denote by $  \mu_j , 1 \le j\le m , $ the \emph{distinct} eigenvalues  
of $A(\underline a)$ and introduce $\delta = \min_{j \ne k} | \mu_j -   \mu_k| $. 
For $a$ in a neighborhood $\omega $  of $\underline a$, the spectrum of $A(a)$ is
contained in nonoverlapping intervals  $I_j = ] \mu_j - \delta /4 ,  \mu_j + \delta /4[$. 
Write
\begin{equation}
\label{decomp}
\begin{aligned}
S_A(a) = \sum_j & S(  \mu_j, \underline a) \Pi_{I_j}(a) 
\\
&+ \sum_j \sum_{ \lambda \in I_j \cap \Sigma(a) } \Big(S(\lambda, a) - S( \mu_j, \underline a)\Big)
\Pi _ \lambda( a) 
\end{aligned}
\end{equation}
with 
 \begin{equation}
\Pi_j(a) = \sum_{   \lambda \in I_j \cap \Sigma(a) } \Pi _ \lambda( a). 
\end{equation}
Using a uniform bound for the $\Pi_\lambda(a)$ and the continuity of the eigenvalues at $\underline a$ one concludes that the second sum  in \eqref{decomp}  tends to $0$ 
as $a $ tends to $\underline a$. By Proposition~\ref{proplipproj}, $\Pi_j$ is continuous  and hence  
  $S_A$ is continuous at $\underline a$. 
\end{proof}

\begin{proof}[{\bf 3)} Lipschitz continuity] 
By Lemma~\ref{lemLip} it is sufficient to prove that there is a positive constant   $K $   such that 
for all  $  \underline a \in \Omega$,  there is a neighborhood $\omega$  of $ \underline a$, such that 
\begin{equation} 
\label{youpi2} 
 \forall  a   \in \omega, \qquad  | S_A(a )  - S_A( \underline a)  | \le K | a - \underline a|. 
\end{equation}
The proof starts as in 2) with the decomposition \eqref{decomp}. Shrinking $\omega $ if necessary, 
the Lipschitz continuity of 
$S$ and of the eigenvalues implies that   for $a \in \omega$ and 
$ \lambda \in I_j \cap \Sigma(a)$ 
$$
\Big| S(\lambda, a) - S(  \mu_j, \underline a)\Big| \le C | a - \underline a|
$$
where $C$ depends only  $  \big\| \nabla_a A \big\|_{L^\infty(\Omega)}$ and
$  \big\| \nabla_{\lambda ,a} S \big\|_{L^\infty(\RR \times \Omega)}$.   
Since the $\Pi_\lambda$ are uniformly bounded, this implies that 
the second sum in \eqref{decomp} is uniformly  $O(|a- \underline a|)$ so that it remains 
to prove that, with $S_j = S(  \mu_j, \underline a)$ and 
$ p_j(a) =  \Pi_{j}(a) - 
\Pi_{j}(\underline a)\big) $,  one has
\begin{equation}
\label{est513}
\Big|   \sum_j   S_j  p_j(a)  \Big|  \le  K | a - \underline a|.
\end{equation}
with $K$ independent of $\underline a$ and $\omega$. Since $S \in W^{1, \infty}$, the  $S_j $ satisfy 
\begin{equation}
\label{hypvarS} 
\big| S_j \big| \le  K , \qquad \big| S_j - S_k \big| \le  K_1 | \mu_j - \mu_k |.  
\end{equation} 
Moreover,  Proposition~\ref{proplipproj} implies that 
for all $J \subset \{ 1, \ldots , m \}$ with $J \neq \emptyset$  and $J \neq \{ 1, \ldots , m \}$, 
\begin{equation}
P_J = \sum_{j \in J}  p_j (a) . 
\end{equation}
satisfies, with $\eps = | a - \underline a|$: 
\begin{equation}
\label{deltaJ} 
\big| P_J \big| \le K_2  \eps  \left(\min_{ j \in J, k \notin J} | \mu_j - \mu_k | \right)^{-1}.  
\end{equation}
Moreover, when  $J   = \{ 1, \ldots , m \}$,  
\begin{equation}
\label{deltaJmax}
 P_{\{ 1, \ldots , m \}}   = 0. 
\end{equation} 
The next lemma implies   \eqref{est513},  finishing the proof of the proposition. 
 \end{proof}
 
\begin{lem}
\label{mainlem} There is a constant $C_m $ which depends only on $m$, such that 
for all   $S_j$ and $p_j$, $1\le j\le m$  satisfying \eqref{hypvarS} \eqref{deltaJ}and  \eqref{deltaJmax}, the 
sum $S = \sum S_j p_j$ satisfies 
\begin{equation}
\label{estS} 
\big|  S \big| \le C_m  K_1 K_2 \eps . 
\end{equation}
\end{lem}
\begin{proof}
By homogeneity, we can assume that $K_1 = K_2  = 1$. 
The proof is by induction on $m$. When $m =1$, the condition  \eqref{hypvarS} reduces  to $|S_{1} | \le 1$, 
the condition \eqref{deltaJ} is void and $ p_{1}   = 0 $.

We assume that the lemma is proved up to order $m- 1 \ge1$ and we prove  it at the order $m$. 
 Let  
\begin{equation}
 \delta :=  \min_{ j \ne  k } | \mu_j - \mu_k |  > 0 . 
\end{equation}
Permuting the indices we can assume that  the infimum is attained 
for $(j, k) = (m-1, m)$, which means that
\begin{equation}
\label{relabel}
\forall j \ne k : \quad  | \mu_{m-1}  - \mu_m |  \le  | \mu_j - \mu_k | . 
\end{equation}
We spilt $S$ in two terms: 
\begin{equation}
\label{splitS}
S = \widetilde S  \   +   \   (S_m - S_{m-1}) p_m  
\end{equation}
with 
\begin{equation}
\label{varS1} 
\widetilde S = \sum_{j= 1}^{m-1}  S_j  \tilde p_j   
\end{equation}
where 
$ \tilde p_j = p_j $   when  $ j \le m - 2$  and   $  \tilde p_{m-1} = p_{m-1} + p_m $.  

The condition \eqref{deltaJ} applied to $ J = \{ m\}$ and \eqref{relabel} imply that 
$| p_m |  \le \eps \delta^{-1} $ while \eqref{hypvarS} and  \eqref{relabel} imply that 
$| S_{m} - S_{m-1} | \le \delta$. This shows that the second term in \eqref{splitS} satisfies 
$| (S_m - S_{m-1})  p_m    | \le \eps $. 

We now check  that 
the induction hypothesis can be applied to $ \widetilde S $. 
The condition \eqref{hypvarS} is clear, so we only have to show 
that the conditions \eqref{deltaJ}
and \eqref{deltaJmax} are satisfied for the $\tilde p_j  $. 

Consider  a non emptyl subset $\widetilde J \subset \{1, \ldots, m-1\}$.  Then
$\widetilde P_{\widetilde J}  = P_{J } $ with

\qquad 1)   $ J  = \widetilde J \cup\{m\}$  if   $m-1 \in \widetilde J$, 

\qquad 2)   $ J  = \widetilde J  $  if   $m-1 \notin \widetilde J$, 

\noindent
In particular, $\widetilde P_{ \{1, \ldots, m-1\}} =  P_{ \{1, \ldots, m\}} = 0$ 
so that \eqref{deltaJmax} for $\widetilde S$ is satisfied.

Suppose that  $\widetilde J \ne  \{ 1, \ldots, m-1\}$ and let  $J$ be as above. Introduce also 
$\widetilde K =  \{1, \ldots, m-1\} \backslash \widetilde J$ and 
$K = \{ 1, \ldots, m \} \backslash J$. One has $K = \widetilde K$ in case 1) and 
$K = \widetilde K \cup\{m \}$ in case 2). By assumption, we know that 
\begin{equation*}
\big| \widetilde P_{\tilde J }   \big| 
= \big|   P_{  J }   \big|  \le  \eps  \left(\min_{ j \in J  , k \in K }
  | \mu_j - \mu_k | \right)^{-1} .
\end{equation*}
We claim that 
\begin{equation}
\label{verif1}
\min_{ j \in \tilde J ,  k \in \tilde K  } | \mu _j - \mu_k |
\le   2 
\min_{ j \in J  , k \in  K   } | \mu_j - \mu_k |. 
\end{equation} 
Indeed, if it is true, it implies that 
\begin{equation}
\label{deltaJt} 
\big| \widetilde P_{\tilde J }   \big| \le   2 \eps  \left(\min_{ j \in \tilde J  
,  k \in \tilde K } | \mu_j - \mu_k | \right)^{-1}
\end{equation}
 so that the induction hypothesis is satisfied for $\widetilde S$ with $\eps$ replaced by $2 \eps$. Thus 
 $ 
 \big| \widetilde S \big| \le  2 C_{m-1} \eps 
 $  and \eqref{estS} follows with $C_m = 1 + 2 C_{m-1}$. 

Therefore, to complete the proof, it remains to prove the claim \eqref{verif1}.  
In this estimate, $\tilde J$ and $\tilde K$ play symmetric roles, and therefore 
we can assume that $m \in J$, that is that $m-1 \in \tilde J$.  
Comparing the sets $\tilde J \times \tilde K$ and $J \times K$, we see that the only nontrivial case concerns 
 $ | \mu_m- \mu_k | $ when $k \in K = \tilde K$. In this case, since 
$m-1 \in \tilde J$, the claim follows from the inequality 
\begin{equation*}
 | \mu_{m-1}- \mu_k |  \le  
  | \mu_{m }- \mu_k |  +  | \mu_{m-1}- \mu_m  |   \le  2  | \mu_m- \mu_k | 
\end{equation*}
where we have used \eqref{relabel}. 
The proof  of the lemma is now complete. 
\end{proof}

\bigbreak



\section{Appendix B : Symmetrizable  matrices} 

\subsection{Positivity of symmetrizers and bounds }
   
  \begin{prop}
  \label{propPSP}
 Suppose that $L$ and $\bS$ ares matrices such that $\bS L$ is hermitian symmetric. 
Suppose that $J $ is  an invertible matrix such that 
\begin{equation}
\label{positiv1}
\forall u \in \ker L  , \qquad \re  \big(  \bS J u ,  u \big)  \ge c | u |^2  . 
\end{equation}
Then, $0$ is a semi-simple eigenvalue of $J^{-1}L  $  and the associated eigenprojector 
$\Pi $ satisfies  
\begin{equation}
\label{eqPiSPi}
\Pi^*   \bS J  \Pi = \Pi^*  \bS  J  
\end{equation}
and 
\begin{equation}
\label{normePi}
\big| \Pi \big| \le  | \Sigma  J  | / c. 
\end{equation}
Moreover, 
\begin{equation}
\label{ortho}
\bS   f   \in   (\ker L) ^\perp \quad \Leftrightarrow \quad  f \in \mathrm{range} (L).
\end{equation}

  \end{prop}

\begin{proof}  
Let $\bK $ and $\bR$ denote respectively the kernel and the range of $L$. 
The identity $(\bS Lu, v) =  (u ,  \bS L v) $  implies  
that $\bS  \bR  \subset \bK^\perp$ and hence $\bR \subset \bS^{-1}( \bK^\perp) $. 
Next we note that \eqref{positiv1} implies that  if 
$u\in \bK$ and $\bS J u \in \bK^\perp$, then $u= 0$, so that
  $ J \bK  \cap \bS^{-1} (\bK^\perp ) = \{0\}$: 
  \begin{equation}
\label{tagada}
\bR \subset \bS^{-1}( \bK^\perp), \qquad J \bK  \cap \bS^{-1} (\bK^\perp ) = \{0\}. 
\end{equation} 
 In particular $J \bK \cap  \bR = \{0\}$ and $\bK \cap J^{-1} \bR = \{0\}$. 
This means that the kernel  $\bK$ of $A_J - \tau \Id$   has a trivial intersection 
with the range $J^{-1} \bR $ of $A_J - \tau\Id$, that is that 
 $\tau $ is a semisimple eigenvalue of $A_J$. 
 Moreover, since $\dim \bR +   \dim  J \bK = N $, \eqref{tagada} implies that 
 $\bR  =  \bS^{-1}( \bK^\perp) $, that is \eqref{ortho}.

In the splitting $u = \Pi u + (\Id- \Pi) u$, $\Pi u \in \bK $ and  there is $v$ 
such that $(\Id- \Pi) u= J^{-1} Lv$. Therefore  $(\bS L v , \Pi  u) = 0$ and 
$\big( \bS J \Pi u, \Pi u) =  \big( \bS J u, \Pi u)$. Hence, 
$$
c \big| \Pi u \big|^2 \le \re \big( \bS J \Pi u, \Pi u) = \re \big( \bS J u, \Pi u)
\le \big| \bS J u \big|  \ \big| \Pi u \big|
$$
and \eqref{normePi} follows.

For  $f \in \bK^\perp$ and $u\in \CC^N$ one has   $(\Pi^* f , u )   =  (  f , \Pi u ) = 0$. 
 Thus  $ \bK^\perp \subset \ker \Pi^*$ and indeed $ \bK^\perp =  \ker \Pi^*$ 
 since the two spaces have the same dimension. 
 For all $u$,  $(\Id - \Pi) u \in J^{-1} \bR$,
 hence   $ \bS J (\Id - \Pi)  u \in \bK^\perp$ and therefore
 $\Pi^* \Sigma J (\Id - \Pi)  u = 0$ that is \eqref{eqPiSPi}. 
\end{proof}

\begin{prop}
\label{positivcone}
Suppose that $L$ and $\bS$ ares matrices such that $\bS L$ is hermitian symmetric. 
We now assume that we are given two matrices 
$J_0$ and $J_1$ such that for all $t \in [0,1]$
$J_t = (1-t) J_0 + t J_1$ is invertible, and for all $u$ and $v$ in $\ker $, 
 \begin{equation}
\label{symlin1}
 \big(  \bS J_t   u ,  v \big) =  \big(  u, \bS (a) J_t  v \big) . 
\end{equation}
Suppose that $J_0$ satisfies 
\begin{equation}
\label{positiv10}
\forall u \in \bK_a , \qquad \big(  \bS  J_0 u ,  u \big)  \ge c | u |^2  . 
\end{equation}
and suppose that for $t \in [0, 1[$,  $0$ is a semi simple eigenvalue of $J_t^{-1} L $. 
Then for all $t \in [0,1]$, 
\begin{equation}
\label{positiv11}
\forall u \in \ker L , \qquad \big(  \bS J_t  u ,  u \big)  \ge (1-t) c | u |^2  . 
\end{equation}
\end{prop}

\begin{proof}
The assumption \eqref{symlin1} means that  the restriction of  $\bS  J_t $ to  $\ker L$ is 
symmetric.  By  \eqref{positiv10}, it is positive definite for $t = 0$. By continuity, it remains positive 
as long as it remains definite. It is indefinite when there is a  $u\in \ker L$, $u \ne 0$  such that 
$\bS J_t   u (\in \ker L)^\perp$.    By \eqref{ortho}, this would imply that 
$ u\ne 0$ would belong both to the kernel and to the range  
of $J_t{-1} L$, contradicting the assumption that $0$ is  a semisimple eigenvalue. 
By continuity, $\bS J_1$ is nonnegative and \eqref{positiv11} follows. 
 \end{proof}



\subsection{From full symmetrizers to symmetrizers}

  We consider here $N \times N$ matrices 
\begin{equation}
L(\tau , a)  =  \tau J(a)  -    A (a) 
\end{equation}
 which depend  on parameters $a$ in an open set $\Omega $ and $\tau \in \RR$.
 We always assume that $J(a) $ is invertible. We link the spectral properties
 of $A_J(a) := J(a)^{-1} A(a)$ to the existence of symmetrizers and full symmetrizers 
 of $L(\tau, a)$.

  \begin{ass}
  \label{assfullsym}
  We assume that the matrices $J(a)$, $J(a)^{-1} $ are uniformly bounded and that 
  there are uniformly bounded matrices 
  $\bS (\tau, a)$  is such that for all $\tau$ and $a$, 
  $\bS(\tau, a) L(\tau, a)$ is hermitian symmetric and 
\begin{equation}
\label{positivbis}
\forall u \in \ker L(\tau, a)  , \qquad \re  \big(  \bS (\tau, a) J(a) u ,  u \big)  \ge c | u |^2  . 
\end{equation}
where $c$ is independent of $a$ and $\tau$. 

We further assume that 
all the complex roots in $\tau$ of $\det L(\tau, a) = 0$ are real.  
  \end{ass}

   Proposition~\ref{propPSP} implies that   eigenvalues $\tau$ of $J(a)^{-1} A(a)$  are real and semi simple and  that  corresponding 
eigenprojectors are uniformly bounded:

 \begin{cor}
\label{cor62}

Under Assumption~\ref{assfullsym},  the family $A_J(a)$ is uniformly strongly hyperbolic in the sense of Assumption~$\ref{assSHM}$. 

\end{cor}

\begin{theo}
\label{theomajmod}
In addition to Assumption~$\ref{assfullsym}$, suppose that  
$J$, $A$ and $\bS$ are continuous [resp. Lipschitz continuous] [resp. $C^\infty$] 
in $a \in \Omega$ and $\tau \in \RR$. Then there is  a bounded and 
continuous [resp. Lipschitz continuous] [resp. $C^\infty$]  matrix $S(\cdot)$ on $\Omega$  such that 
\begin{equation}
S(a) J(a) = \big( S(a) J(a) \big)^* \ge c_1 \Id , \qquad  S(a) A (a) = \big( S(a) A(a) \big)^*, 
\end{equation} 
with $c_1 > 0$ independent of $a$.

\end{theo}

\begin{proof} Multiplying $L$ by $J^{-1}$ reduces to the case $J = \Id$. 
A symmetrizer is 
\begin{equation}
\label{defiS}
S(a) =  \sum_{\tau \in \Sigma(a) } \Pi(\tau, a)^* \bS (\tau, a ) \Pi(\tau, a) 
\end{equation} 
where $\Sigma(a) \subset \RR $ denotes the spectrum of 
$A(a)$. 

For $u$ and $v$ in $\ker L(\tau, a)$, one has $L(\tau + \sigma s , a)   u = \sigma  J(a) u $ and a similar expression for $v$. 
The symmetry implies 
\begin{equation*}
\big( \bS(\tau + \sigma  \nu) J (a) u , v \big)  = \big( u,  \bS(\tau + \sigma u)J (a) v \big)    
\end{equation*} 
and letting $\sigma$ tend to zero implies 
 \begin{equation}
\label{symlin2}
 \big(  \bS(\tau, a) J(a) u ,  v \big) =  \big(  u,  \bS (\tau, a) J(a)  v \big) . 
\end{equation}

This shows that each term of the sum \eqref{defiS} is symmetric and 
$S$ is symmetric. Moroever, by \eqref{positivbis} and Corollary~\ref{cor62}, it is   uniformly bounded and  uniformly positive. 
By construction, $S(a) A(a)$ is symmetric and  
by \eqref{eqPiSPi} and symmetry,  one has 
\begin{equation}
S(a) =  \sum_{\tau \in \Sigma(a) } \Pi(\tau, a)^* \bS (\tau, a ) = 
\sum_{\tau \in \Sigma(a) }  \bS^*  (\tau, a ) \Pi(\tau, a). 
\end{equation} 
Theorem~\ref{theoFC}  implies that $S $ is continuous [resp. Lipschitz continuous] 
[resp. $C^\infty$] in $a$, finishing the proof of the theorem. 
\end{proof} 








 \end{document}